\documentclass[12pt]{amsart}
\usepackage{}
\usepackage{amsmath,amssymb,txfonts}
\usepackage{amssymb}
\usepackage{amsxtra}
\usepackage{amsthm, color}
\usepackage{txfonts}
\usepackage{graphicx}
\usepackage{times}

\def\rr{{\mathbb R}}
\def\rn{{{\rr}^n}}

\def\nn{{\mathbb N}}
\def\zz{{\mathbb Z}}

\def\cd{{\mathcal D}}

\def\ch{{\mathcal H}}

\def\cs{{\mathcal S}}

\def\fz{\infty}
\def\az{\alpha}

\def\supp{{\mathop\mathrm{\,supp\,}}}
\def\dist{{\mathop\mathrm{\,dist\,}}}

\def\loc{{\mathop\mathrm{\,loc\,}}}

\def\BMO{{\mathop\mathrm{\,BMO\,}}}

\def\lz{\lambda}

\def\ez{\epsilon}

\def\bz{\beta}

\def\kz{{\kappa}}
\def\tz{\theta}

\def\wz{\widetilde}

\def\ls{\lesssim}

\def\laz{\langle}
\def\raz{\rangle}

\def\r{\right}
\def\lf{\left}

\def\XXint#1#2#3{{\setbox0=\text{$#1{#2#3}{\int}$ }
\vcenter{\text{$#2#3$ }}\kern-.6\wd0}}

      \newtheorem{theorem}{Theorem}[section]
      \newtheorem{definition}[theorem]{Definition}

      \newtheorem{proposition}[theorem]{Proposition}
      \newtheorem{lemma}[theorem]{Lemma}

      \newcommand{\lt}[1]{[ {#1}] \lower.3ex\text{$_{t}$}}

\begin{document}

\title[Intrinsic nature of the Stein-Weiss $H^1$-inequality]
{Intrinsic nature of the Stein-Weiss $H^1$-inequality}

\author{Liguang Liu}
\address{School of Mathematics,
Renmin University of China,
Beijing 100872, China}
\email{liuliguang@ruc.edu.cn}

\author{Jie Xiao}
\address{Department of Mathematics and Statistics,
Memorial University, St. John's, NL A1C 5S7, Canada}
\email{jxiao@math.mun.ca}


\thanks{
LL was supported by the National Natural Science Foundation of China
(\# 11771446);
JX was supported by NSERC of Canada (\# 202979463102000).
}


\date{\today}


\begin{abstract}
This paper explores the intrinsic nature of the celebrated Stein-Weiss $H^1$-inequality
	$$
	\|I_s u\|_{L^\frac{n}{n-s}}\lesssim \|u\|_{L^1}+\|\vec{R}u\|_{L^{1}}=\|u\|_{H^1}
	$$
through	the tracing and duality laws based on Riesz's singular integral operator $I_s$. We discover that $f\in I_s\big([\mathring{H}^{s,1}_{-}]^\ast\big)$
if and only if $\exists\ \vec{g}=(g_1,...,g_n)\in \big(L^\infty\big)^n$
such that $f=\vec{R}\cdot\vec{g}=\sum_{j=1}^n R_jg_j$ in $\mathrm{BMO}$ (the John-Nirenberg space introduced in their 1961 {\it Comm. Pure Appl. Math.} paper \cite{JN})
where $\vec{R}=(R_1,...,R_n)$ is the vector-valued Riesz transform - this characterizes the Riesz transform part $\vec{R}\cdot\big(L^\infty\big)^n$ of Fefferman-Stein's decomposition  (established in their 1972 {\it Acta Math} paper \cite{FS}) for $\mathrm{BMO}=L^\infty+\vec{R}\cdot\big(L^\infty\big)^n$ and yet indicates that $I_s\big([\mathring{H}^{s,1}_-]^\ast\big)$
is indeed a solution to Bourgain-Brezis' problem under $n\ge 2$:
``What are the function spaces $X, W^{1,n}\subset X\subset \mathrm{BMO}$,
such that every $F\in X$ has a decomposition $F=\sum_{j=1}^n R_j Y_j$ where $Y_j\in L^\infty$?" (posed in their 2003 {\it J. Amer. Math. Soc.} paper \cite{BB}).
\end{abstract}

\maketitle

\tableofcontents

\arraycolsep=1pt
\numberwithin{equation}{section}

\section{Introduction}\label{s1}

\subsection{The Stein-Weiss $H^p$-inequalities}\label{s11}
For $(n,p)\in\mathbb N\times[1,\infty)$, denote by $H^p$ the real Hardy space on the Euclidean space $\rn$, consisting  of all functions $f$ in the Lebesgue space $L^p$ with
$$\|u\|_{H^p}=\|u\|_{L^p}+\|\vec{R}u\|_{L^p}<\infty,$$
where $$\vec{R}=(R_1,\dots, R_n)$$  is the vector-valued Riesz transform on $\rn$, with
$$
\vec{R}u=(R_1u,...,R_nu)\ \ \&\ \
R_ju(x)=\left(\frac{\Gamma(\frac{n+1}{2})}{\pi^{\frac{n+1}{2}}}\right)\;\text{p.v.}\int_{\mathbb R^n}\frac{x_j-y_j}{|x-y|^{n+1}}u(y)\,dy
\ \ \textup{a.\,e.}\ x\in\rn$$
and $\Gamma$ being the Gamma function. Also, for a vector-valued function
$$\vec{f}=(f_1,\dots, f_n)$$
let
$$
\|{\vec{f}}\|_{L^p}=\sum_{j=1}^n \|f_j\|_{L^p}.
$$
Note that
$H^p$ coincides with the classical Lebesgue space $L^p$ whenever $p\in(1,\infty)$ and the $(0,1)\ni s$-th order Riesz singular integral operator $I_s$ acting on a function
$$u\in \bigcup_{{p\in[1,\frac ns)}}L^p$$ is defined by
$$
I_su(x)=\lf(\frac{\Gamma(\frac{n-s}{2})}{\pi^\frac{n}{2}2^s\Gamma(\frac{s}{2})}\r)\int_{\mathbb R^n}|x-y|^{s-n} u(y)\,dy
\qquad\textup{a.\,e.}\ x\in\rn.
$$
We refer the reader to Stein's seminal texts \cite{St, St2} for more about these basic notions.
The well-known Stein-Weiss $H^p$-inequality (cf. \cite{SW}) states that under
$$
0<s<1\ \ \&\ \
1\le p< \frac ns,
$$
the Riesz-Hardy potential space $I_s(H^p)$ can be continuously embedded into
$L^\frac{pn}{n-sp}$, that is,
\begin{equation}
\label{e1}
\|I_s u\|_{L^\frac{pn}{n-sp}}\lesssim \|u\|_{L^p}+\|\vec{R}u\|_{L^p}\approx\|u\|_{H^p} \qquad \ \forall\ \ u\in  H^{p}.
\end{equation}

Let $C_c^\infty$ be the collection of all infinitely differentiable functions compactly supported in $\rn$.
Note that $C_c^\infty\cap H^p$ is dense in $H^p$ for any $p\in[1,\infty)$.
For any $u\in C_c^\infty$ let
$$
(-\Delta)^\frac{s}{2}u(x)=\begin{cases}I_{-s}u(x)=c_{n,s}\int_{\mathbb R^n}\frac{u(x+y)}{|y|^{n+s}}\,dy\ \ &\text{as}\ \ s\in (-1,0)\\
u(x)  \ \ &\text{as}\ \ s=0\\
c_{n,s,+}\int_{\mathbb R^n}\frac{u(x+y)-u(x)}{|y|^{n+s}}\,dy\ \ &\text{as}\ \ s\in (0,1)
\end{cases}
$$
and
$$
\nabla^su(x)=\Bigg(\frac{\partial^s u}{\partial x_j^s}\Bigg)_{j=1}^n=\vec{R}(-\Delta)^\frac{s}{2}u(x)=c_{n,s,-}\int_{\mathbb R^n}\frac{y\big(u(x)-u(x-y)\big)}{|y|^{n+1+s}}\,dy,
$$
where (cf. \cite[Definition~1.1, Lemma~1.4]{Bucur} for $c_{n,s,+}$ and \S \ref{s2} below for $c_{n,s,-}$)
$$
\begin{cases}
c_{n,s}=\frac{\Gamma(\frac{n-s}{2})}{\pi^\frac{n}{2}2^s\Gamma(\frac{s}{2})}\\
c_{n,s,+}=\frac{s2^{s-1}\Gamma\big(\frac{n+s}{2}\big)}{\pi^{\frac{n}{2}}\Gamma\big(1-\frac{s}{2}\big)}\\
c_{n,s,-}=\frac{2^{s}\Gamma\big(\frac{n+s+1}{2}\big)}{\pi^{\frac n2}\Gamma\big(\frac{1-s}{2}\big)}.
\end{cases}
$$
In particular, if $0<s<n=1$ then there are two $s$-dependent constants $c_{\pm}$ to make the following Liouville fractional derivative formulas (cf. \cite{SS2}):
$$
\begin{cases}
(-\Delta)^\frac{s}{2}u(x)=c_+\bigg(\frac{d^s}{dx^s_+}+\frac{d^s}{dx^s_-}\bigg)u(x)\\
\nabla^su(x)=c_-\bigg(\frac{d^s}{dx^s_+}-\frac{d^s}{dx^s_-}\bigg)u(x)\\
\frac{d^s}{dx^s_\pm}u(x)=\frac{s}{\Gamma(1-s)}\int_{\pm\infty}^0\frac{t(u(x+t)-u(x))}{|t|^{2+s}}\,dt.
\end{cases}
$$
Hence it is natural and reasonable to adopt the notations
$$
\nabla^s_+u=(-\Delta)^\frac{s}{2}u\qquad \&\qquad \nabla^s_-u=\nabla^su=\vec{R}(-\Delta)^\frac{s}{2}u.
$$
The operators $\nabla^s_+$ and  $\nabla^s_-$ can be viewed as the fractional extensions of the gradient operator
$$\nabla=(\partial_{x_1},\dots,\partial_{x_n}).$$
Accordingly, for any $s\in(0,1)$,
the  Stein-Weiss inequality \eqref{e1} (cf. \cite{SSS17}) amounts to
\begin{equation}
\label{e3}
\|u\|_{L^\frac{pn}{n-sp}}\lesssim \|\nabla^s_+u\|_{L^p}+\|\nabla^s_{-}u\|_{L^p} \qquad\forall\ \ u\in I_s(C_c^\infty\cap H^p).
\end{equation}
Of course, it is appropriate to mention the following basic facts:
\begin{itemize}

\item[$\rhd$] If $0<s<1<p<n/s$, then the right-hand-side  of \eqref{e3} can be replaced by $\|\nabla^s_\pm u\|_{L^p}$.
More precisely, on the one hand, the boundedness of $\vec R$ on $L^{p>1}$ and \eqref{e3} give (cf. \cite[Lemma 2.4]{SS1})
\begin{equation*}
\|u\|_{L^\frac{pn}{n-sp}}\lesssim \|\nabla^s_{+}u\|_{L^{p}}\qquad\forall\ \  u\in I_s(C_c^\infty\cap H^p).
\end{equation*}
One the other hand, \cite[Theorem 1.8]{SS1} derives
\begin{equation*}
\|u\|_{L^\frac{pn}{n-sp}}\lesssim \|\nabla^s_{-}u\|_{L^{p}}\qquad\forall\ \  u\in I_s(C_c^\infty\cap H^p).
\end{equation*}

\item[$\rhd$] If $0<s<p=1<n$, then the right-hand-side of \eqref{e3} can be replaced by $\|\nabla^s_{-}u\|_{L^{1}}$ {(cf. \cite[Theorem $A'$]{SSS17})}  but cannot be replaced by $\|\nabla^s_+u\|_{L^1}$ (cf. \cite[p.119]{St}).

\item[$\rhd$] If $0<s<p=1=n$, then the right-hand-side of \eqref{e3} cannot be replaced by either $\|\nabla^s_+ u\|_{L^{1}}$ or
$\|\nabla^s_- u\|_{L^{1}}$. A counterexample
is given in \cite[Section~3.3]{SSS17}.

\item[$\rhd$] If  $0<s<p=1\le n$, then instead of the strong-type estimates, one has the following weak-type inequality:
\begin{equation*}
\|u\|_{L^{\frac{n}{n-s},\infty}}=\sup_{t>0}t\big|\big\{x\in\rn:\ |u(x)|>t\big\}\big|^\frac{n-s}{n}\lesssim \|\nabla^s_{\pm}u\|_{L^{1}}
\qquad\forall\ \  u\in I_s(C_c^\infty\cap H^1),
\end{equation*}
while the case for $\|\nabla^s_{+}u\|_{L^{1}}$ is due to the boundedness of $I_s$ from $L^1$ to $L^{\frac{n}{n-s},\infty}$ (cf. \cite{Aduke} or \cite[p.\,119]{St}) and for $\|\nabla^s_{-}u\|_{L^{1}}$ follows further from  \cite[(1.5)]{O} showing
\begin{equation*}
\text{id}=-\sum_{j=1}^n R_j^2\ \ \&\ \
\|R_ju\|_{L^{\frac{n}{n-s},\infty}}\lesssim \|u\|_{L^{\frac{n}{n-s},\infty}}\qquad \forall\ \ (j,u)\in\{1,2,\dots,n\}\times L^{\frac{n}{n-s},\infty}.
\end{equation*}
\end{itemize}

\subsection{Overview of the principal results}\label{s12}
The above analysis has driven us to take a fractional-geometrical-functional look at the most important case $p=1$ of the Stein-Weiss inequality \eqref{e1}.

\subsubsection*{Dense subspaces of $H^{s,1}\ \&\ H^{s,1}_\pm$}\label{s1.2}

Denote by $\cs$ the Schwartz class on $\rn$ consisting of functions $f\in C^\infty$ such that
$$
\rho_{N,\az}(f)=\sup_{x\in\rn}(1+|x|^N)|D^\az f(x)|<\infty\ \ \text{holds for}\ \
\begin{cases}
N\in{\zz_+}=\nn\cup\{0\}\\
\az=(\az_1,\dots,\az_n){\in\zz_+^n}\\
 D^\az=\partial_{x_1}^{\az_1}\cdots\partial_{x_n}^{\az_n}.
 \end{cases}
 $$
Also, write $\cs'$ for the Schwartz {tempered} distribution space - the dual of $\cs$ endowed with the weak-$\ast$ topology.

As detailed in \S \ref{s2}, given $s\in(0,1)$, if
we let
$$\cs_s=\lf\{f\in C^\infty:\  \rho_{n+s,\az}(\phi)=\sup_{x\in\rn}(1+|x|^{n+s})|D^\az f(x)|<\infty\;\ \forall\ \alpha\in\zz_+^n\r\},$$
then for any
$$u\in\cs_s'\subset\cs'$$ we can
define $\nabla^s_\pm u$ as a distribution in $\cs'$.
This definition and the case $p=1$ of \eqref{e3} motivate us to consider the
fractional Hardy-Sobolev space
$$
H^{s,1}=\lf\{u\in\cs_s':\ [u]_{H^{s,1}}=\|(-\Delta)^\frac s2 u\|_{H^1}<\infty\r\}.
$$
Note that
$$u_1-u_2=\text{constant}\Leftrightarrow [u_1]_{H^{s,1}}=[u_2]_{H^{s,1}}.
$$
So, $[\cdot]_{H^{s,1}}$ is properly a norm on quotient space of $H^{s,1}$ modulo the space of all real constants,
and consequently this quotient space is {a Banach space.} 

Upon introducing
$$
H^{s,1}_\pm=\lf\{u\in\cs_s':\ [u]_{H^{s,1}_\pm}=\|\nabla^s_\pm u\|_{L^1}<\infty\r\},
$$
we immediately find
$$
H^{s,1}=H^{s,1}_+\cap H^{s,1}_{-}.
$$
Also, since $\cs$ is dense in $H^{s,1}$ {but} it is hard to see the density of $\cs$ in $H^{s,1}_\pm$, we are induced to introduce
\begin{align*}
\mathring{H}^{s,1}_\pm
&=\text{closure of}\ \cs\ \text{in}\ H^{s,1}_\pm\ \text{under}\  [\cdot]_{H^{s,1}_\pm},
\end{align*}
and yet still have
$$
H^{s,1}=\mathring H^{s,1}_+\cap \mathring H^{s,1}_-
$$
whose $\mathring{H}^{s,1}_\pm$ is a Banach space modulo the space of all real constants.

Correspondingly, for $s\in(0,1)$ let
$W^{s,1}$ be the
collection of all locally integrable functions $u$ on $\rn$ obeying
$$
[u]_{W^{s,1}}= \int_{\mathbb R^n}\int_{\mathbb R^n}\frac{|u(x)-u(y)|}{|x-y|^{n+s}}\,dy\,dx<\infty.
$$
Then the quotient space of $W^{s,1}$ modulo the space of all real constants is equal to the homogeneous Besov space $\dot\Lambda^s_{1,1}$ (cf. \cite{Xadv}) and is also called Sobolev-Slobodeckij space (cf. \cite[p.\,36]{Tr83}) or fractional Sobolev space (cf. \cite{PS}), and hence
$$\cs_\infty=\lf\{f\in\cs:\ D^\alpha\hat f(0)=0\ \forall\ {\az\in\zz_+^n}\r\}$$
is dense in $W^{s,1}$.
In accordance with \cite[Appendix]{CS}, any function
$$f\in L^1\cap W^{s,1}$$ can also be approximated by
functions in $C_c^\infty$. Since (cf. \cite{SS1, SS2})
\begin{equation}
\label{e2}
|\nabla^s_\pm u(x)|\lesssim\int_{\mathbb R^n}\frac{|u(x)-u(y)|}{|x-y|^{n+s}}\,dy\qquad\forall\ \ (u,x)\in {\cs}\times\rn,
\end{equation}
it follows that
\begin{align}\label{e2a}
[u]_{H^{s,1}}=\|\nabla^s_+u\|_{L^1}+\|\nabla^s_{-}u\|_{L^1}\ls [u]_{{W}^{s,1}}\ \ \forall\ \ u\in{\cs}.
\end{align}
Thus,  both $H^{s,1}$ and $\mathring{H}^{s,1}_\pm$ contain  $W^{s,1}$. More information on $\{\nabla^s_\pm, H^{s,1}, H^{s,1}_\pm\}$ is demonstrated in Propositions \ref{prop1}-\ref{prop-dense1}-\ref{prop-dense3}-\ref{prop-dense2}.

\subsubsection*{Tracing laws for $H^{s,1}\ \&\ H^{s,1}_\pm$}\label{s1.3}
The previous discussions derive that
\begin{align}\label{e4x}
\|u\|_{L^\frac{n}{n-s}}\lesssim
\begin{cases}
[u]_{H^{s,1}}\qquad \textup{under}\ \ 0<s<1\le n\\
 [u]_{H^{s,1}_-}\qquad \textup{under}\ \ 0<s<1< n
\end{cases}
\qquad \ \forall\ \  u\in\cs
\end{align}
and
\begin{equation}\label{e4xx}
\|u\|_{L^{\frac{n}{n-s},\infty}}\lesssim
\begin{cases}
  [u]_{H^{s,1}_-}\qquad \textup{under}\ \ 0<s<1=n\\
 [u]_{H^{s,1}_+}\qquad \textup{under}\ \ 0<s<1\le n
\end{cases}
\qquad \ \forall\ \  u\in\cs
\end{equation}
are valid, but
\begin{equation}\label{e4}
\|u\|_{L^\frac{n}{n-s}}\lesssim
\begin{cases}
[u]_{H^{s,1}_-}\qquad \textup{under}\ \ 0<s<1=n\\
[u]_{H^{s,1}_+}\qquad \textup{under}\ \ 0<s<1\le n
\end{cases}
\qquad \ \forall\ \  u\in\cs
\end{equation}
is not true.
In order to understand an essential reason for the truth of \eqref{e4x} or \eqref{e4xx} and the fault of \eqref{e4},
we investigate under what condition of a given nonnegative Radon measure $\mu$ (restricting/tracing a function to a lower dimensional manifold) in $\mathbb R^n$ one has
\begin{equation}
\label{e5}
[u]_X\gtrsim\begin{cases}\|u\|_{L^\frac{n}{n-s}(\mu)}\ \ &\text{as}\ \ X=H^{s,1}\\
\|u\|_{L^{\frac{n}{n-s},\infty}(\mu)}\ \ &\text{as}\ \ X=H^{s,1}_\pm
\end{cases}
\qquad \ \forall\ \ u\in \mathcal{S}\ ?
\end{equation}
Accordingly, we discover such a tracing law that \eqref{e5} is valid if and only if the isocapacitary inequality
\begin{equation}
\label{e6}
\big(\mu(K)\big)^\frac{n-s}{n}\lesssim \text{Cap}_{X}(K)\qquad\forall\ \text{compact}\ \ K\subset\mathbb R^n
\end{equation}
holds, where the right quantity of \eqref{e6} is called $\big\{H^{s,1},H^{s,1}_\pm\big\}\ni X$-capacity of $K$ and defined by
\begin{align*}
\inf\Big\{[f]_{X}:\ 1\le f\ \text{on}\ K\ \&\ \ f\in \mathcal{S}\Big\}.
\end{align*}
In \S \ref{s3}, we utilize the fractional Sobolev capacity $\text{Cap}_{W^{s,1}}$ and the Hausdorff capacity $\Lambda^{n-s}_{(\infty)}$ to handle $\text{Cap}_{X\in\{H^{s,1}, H^{s,1}_\pm\}}$ and its strong or weak capacitary inequality through Theorems \ref{thm3.1} \& \ref{thm3.2}-\ref{t31}. Then, we verify \eqref{e5}$\Leftrightarrow$\eqref{e6} in Theorem \ref{thm1}.

\subsubsection*{Duality laws for $H^{s,1}\ \&\ \mathring H^{s,1}_\pm$}\label{s1.4}
As a by-product of \eqref{e2}-\eqref{e2a} and the capacity analysis developed within \S \ref{s3}, Theorem \ref{thm2} shows that
the dual space $[H^{s,1}]^\ast$ can be characterized by the bounded solutions
$$(U_0,U_1,\dots,U_n)\in\big( L^\infty\big)^{1+n}$$ of the fractional differential equation
$$[\nabla^s_+]^\ast U_0+[\nabla^s_-]^\ast(U_1,\dots,U_n)=T,$$
where
$$[\nabla^s_+]^\ast=(-\Delta)^\frac s2\ \ \ \&\ \ \ [\nabla^s_-]^\ast=-(-\Delta)^\frac s2 \vec{R}.
$$
Also, a similar characterization for
$$
[\mathring H^{s,1}_+]^\ast\ \ \text{or}\ \ [\mathring H^{s,1}_-]^\ast
$$
is presented in Theorem \ref{thm2} in terms of the bounded solutions to the fractional differential equation
$$[\nabla^s_+]^\ast U_0=T\ \ \ \text{or} \ \ \  [\nabla^s_-]^\ast(U_1,\dots,U_n)=T.
$$

Furthermore, suppose that $\BMO$ is the well-known John-Nirenberg class of all locally integrable functions $f$ on $\rn$ with bounded mean oscillation (cf. \cite{JN})
$$
\|f\|_\BMO=\sup_{B\subset\rn} \frac1{|B|}\int_B |f(x)-f_B|\, dx<\infty
$$
where
$$
f_B=\frac1{|B|}\int_B f(x)\,dx
$$
and the supremum is taken over all Euclidean balls $B\subset\rn$ with volume $|B|$. Surprisingly and yet naturally, the argument for Theorem \ref{thm2}, plus the intrinsic structure of
$$
[H^{s,1}]^\ast\ \ \&\ \
[\mathring H^{s,1}_-]^\ast,
$$
reveals (cf. Theorem \ref{thm3}) the inclusions
\begin{equation}
\label{e110}
\BMO\supset I_s\big([\mathring H^{s,1}_-]^\ast\big)=\vec{R}\cdot\big(L^\infty\big)^{n}\supset I_s(L^{\frac ns,\infty})\supset I_s(L^{\frac ns})\ \supset W^{1,n}\ \text{under}\ n\ge 2
\end{equation}
in the sense of $\textup{in}\ \cs'/\mathcal P$ or $\BMO$,
where $L^{\frac ns,\infty}$ denotes the weak Lebesgue $\frac{n}{s}$-space consisting of all Lebesgue measurable functions 
$f$ on $\rn$ such that 
$$\|f\|_{L^{\frac ns,\infty}}=\sup_{t\in(0,\fz)}t|\{x\in\rn: \ |f(x)|>t\}|^{\frac sn}<\infty$$
and ${W}^{1,n}$ is 
the 
space of all locally integrable functions $f$ with respect to  $$
\|f\|_{W^{1,n}}=\left(\int_{\rn}|\nabla f(x)|^n\,dx\right)^\frac1n<\infty,
$$
and one has the following  decomposition of the canonical $\BMO$-function (cf. \cite{Sp, JN})
$$
\begin{cases}
\ln|x|=\sum_{j=1}^n R_j\left(\Bigg(\frac{\Gamma(\frac{1}{2})\pi^{\frac{1-n}{2}}2^{1-n}}{(n-1)\Gamma(\frac{n-1}{2})}\Bigg)\bigg(\frac{x_j}{|x|}\bigg)\right)\ &\text{under}\ n\ge 2;\\
\ln\big|\frac{x+1}{x-1}\big|=\pi H(1_{[-1,1]})(x)=\hbox{p.v.}\int_{\mathbb R}\frac{1_{[-1,1]}(y)}{x-y}\,dy &\text{under}\ n=1,
\end{cases}
$$
with $1_{[-1,1]}$ being the characteristic function of the interval $[-1,1]$. Let us  take the space $I_s(L^\frac ns)$ for example to explain why we require the validity of  \eqref{e110} in the sense of $\cs'/\mathcal P$. Indeed, if $f\in L^\frac ns$, then $I_sf$ may not be pointwisely well defined, but it is a well defined distribution in $\cs'/\mathcal P$.

Nevertheless, the importance of \eqref{e110} can be also seen below.

\begin{itemize}

\item[$\rhd$] Via a different approach, \eqref{e110} reveals the essential structure of the Riesz transform part $\vec{R}\cdot\big(L^\infty\big)^{n}$ of the Fefferman-Stein decomposition (cf. \cite[Theorems 2\&3]{FS}, \cite{U} for a constructive proof, and \cite{Janson, Frazier} for some related discussions):
$$
\BMO=L^\infty+\vec{R}\cdot\big(L^\infty\big)^{n}=L^\infty+I_s\big([\mathring H^{s,1}_-]^\ast\big).
$$
And yet, it is uncertain that $I_s\big([\mathring H^{s,1}_-]^\ast\big)$ is strictly contained in $\BMO$ under $n\ge 2$.

\item[$\rhd$] \eqref{e110} may be treated as a solution to the Bourgain-Brezis question (cf. \cite[p.396]{BB}) - \emph{What are the function spaces $X, W^{1,n}\subset X\subset \BMO$, such that every $F\in X$ has a decomposition $F=\sum_{j=1}^n R_j Y_j$ where $Y_j\in L^\infty$?}. This question is motivated by the fact that ${W}^{1,n}$ obeys the following decomposition (\cite[p.305]{BB})
$$
{W}^{1,n}=\vec{R}\cdot\big(L^\infty\cap{W}^{1,n}\big)^n\ \ \text{under}\ \ n\ge 2.
$$

\item[$\rhd$] As proved in \cite[Theorem 3.5]{PT} (solving the open problem in \cite[Remark 3.12]{AFP}), if $\mathrm{BV}$ is the space of all $L^1$-functions with bounded variation on $\mathbb R^n$, then its dual space $[\mathrm BV]^\ast$ comprises all tempered distributions
$$
f=\nabla\cdot(U_1,\dots,U_n)=\sum_{j=1}^n\partial_{x_j}U_j
\ \ \textup{for some}\ \ (U_1,\dots, U_n)\in \big(L^\infty\big)^n,
$$
and hence
$$
\begin{cases}
[\mathrm BV]^\ast=\nabla\cdot(L^\infty)^n;\\
I_1[\mathrm BV]^\ast=I_1\nabla \cdot (L^\infty)^n=\vec R\cdot (L^\infty)^n\ \text{in}\ \cs'/\mathcal P.
\end{cases}
$$
This indicates that \eqref{e110} has a limiting case $s\uparrow 1$:
$$
W^{1,n}\subset I_1[\mathrm BV]^\ast\subset \BMO\ \ \text{under}\ \ n\ge 2.
$$

\item[$\rhd$] \eqref{e110} improves \cite[Corollary~1.5]{Sp18} which proves that $$I_s(L^{\frac ns,\infty})\subset \vec{R}\cdot\big(L^\infty\big)^{n}.$$

\item[$\rhd$] \eqref{e110} derives that (cf. Theorem \ref{thm3}) for $$(Y_0,n-1)\in I_s\big([\mathring H^{s,1}_-]^\ast\big)\times\mathbb N$$
one can get a vector-valued function
$$
(Y_1,\dots,Y_n)\in \big(L^\infty\big)^n\ \ \text{solving}\ \
\text{div}\big((-\Delta)^{-\frac12}Y_1,\dots, (-\Delta)^{-\frac12}Y_n\big)=Y_0\ \ \text{in}\ \ \cs'/\mathcal P.
$$
Consequently, this divergence-equation-result is valid for
$$
(Y_0,n-1)\in W^{1,n}\times\mathbb N.
$$
Although $W^{1,n=1}$ is a proper subspace of $L^\infty$ on $\mathbb R$, the foregoing consequence cannot be extended to
$$W^{1,\infty}=\Big\{f:\ f\in L^\infty\ \ \&\ \ \nabla f\in (L^\infty)^n\Big\}$$
in the following sense that (cf. \cite[p.394]{BB} or \cite{Mc})
$$
\exists\ \ F_0\in L^\infty\ \ \text{such that}\ \
\text{div}\vec F=F_0\ \ \text{has no solution}\ \  \vec F=(F_1,\dots,F_n)\in \big(W^{1,\infty}\big)^n\ \ \text{under}\ \ n\ge 2.
$$
\end{itemize}

\subsubsection*{Notation}\label{s16} In the foregoing and forthcoming discussions,
$U\lesssim V$ (resp.\, $U\gtrsim V$) means $U\le cV$ (resp. $U\ge cV$) for a positive constant $c$ and $U\approx V$ amounts to $U\gtrsim V\gtrsim U$. Moreover, $1_E$ stands for the characteristic function of a set $E\subset\rn$, and
$$\begin{cases}\nn=\{1,2,\dots\}\\
\zz_+=\{0,1,2,\dots\}\\
\zz=\{0,\pm1,\pm2,\dots\}.
\end{cases}
$$


\section{Dense subspaces of $H^{s,1}\, \&\, H^{s,1}_\pm$}\label{s2}
\setcounter{equation}{0}

\subsection{Initial definitions of $\nabla^s_\pm$}\label{s21}
Note that any $f\in\cs$ has its Fourier transform
$$
\hat f(\xi)=\int_\rn f(x)e^{-2\pi i x\cdot\xi}\,dx\qquad\forall\ \ \xi\in\rn.
$$
So the Fourier transform can be naturally extended to $\cs'$ by the dual paring
$$\laz \hat f,\varphi\raz=\laz f,\hat\varphi\raz \qquad \ \forall\ \ f\in\cs'\;\&\; \varphi\in\cs.$$

\begin{definition}\label{defn0}
For $(s,\phi)\in(-n,1]\times\cs$ let $(-\Delta)^{\frac s2} \phi$ be determined by the Fourier transform
$$\lf((-\Delta)^{\frac s2}\phi\r)^\wedge(\xi)= (2\pi|\xi|)^s\,\hat \phi(\xi)\qquad \forall\ \ \xi\in\rn.
$$
Then we have the following comments.
\begin{enumerate}
\item[\rm (i)] Since $|\xi|^s$ has singularity at the origin, it is not true that
$(-\Delta)^{\frac \beta2} \phi\in\cs$ for general $\phi\in\cs$. However,
if
$$\cs_s=\lf\{f\in C^\infty:\  \rho_{n+s,\az}(\phi)=\sup_{x\in\rn}(1+|x|^{n+s})|D^\az f(x)|<\infty\;\forall\; \alpha\in\zz_+^n\r\},$$
then (cf. \cite[Section~2]{Sil} or \cite{Bucur})
$$(-\Delta)^{\frac s2} \phi\in\cs_s\qquad \forall\ \ \phi\in\cs.$$

\item[\rm (ii)]  Recall that
$$
\cs_\infty=\lf\{f\in\cs:\, D^\alpha\hat f(0)=0\ \ \forall\ \ \az\in\mathbb Z^n_+\r\}.
$$
Then
$$(-\Delta)^{\frac s2} \phi\in\cs_\infty\qquad \forall\ \ \phi\in\cs_\infty.$$
\item[\rm (iii)] As the dual  space of $\cs_\infty$ let $\cs'/\mathcal P$ be  the space $\mathcal S'$ modulo
the space $\mathcal P$ of all real-valued polynomials.
Then, for any $f\in\cs'/\mathcal P$ we can define $(-\Delta)^{\frac s2} f$ as a distribution in $\cs'/\mathcal P$:
$$\laz (-\Delta)^{\frac s2} f,\phi\raz=\laz f, (-\Delta)^{\frac s2} \phi\raz\quad\forall\ \ \phi\in\cs_\infty.$$
Evidently, $(-\Delta)^{\frac s2}$ maps $\cs'/\mathcal P$ onto $\cs'/\mathcal P$ (cf. \cite[pp. 241-242]{Tr83}).
\end{enumerate}
\end{definition}

The $(0,n)\ni \az$-th order Riesz potential $I_\az$ is defined by
$$
I_\az=(-\Delta)^{-\frac \az 2}.
$$
If $f\in\cs$, then $I_\az f$ has the integral expression (cf. \cite[p.\,117]{St})
$$
I_\az f(x)=c_{n,\az}\int_\rn |x-y|^{\az-n}f(y)\, dy\ \ \text{with}\ \
c_{n,\az}=\frac{\Gamma(\frac{n-\az}{2})}{\pi^{\frac n2}2^\az\Gamma(\frac \az2)}.
$$
Based on Definition \ref{defn0}(iii), the definition of $I_\alpha f$ is extendable to $f\in\cs'/\mathcal P$ and so
$I_\az$ maps $\cs'/\mathcal P$ onto $\cs'/\mathcal P$.

\subsubsection*{About $\nabla^s_+$}\label{s211} Upon following \cite[Section~2.1]{Sil}, we can extend the definition of $\nabla^s_+$ to more general distributions.

\begin{definition}\label{defn1}
For $(s,f)\in(0,1)\times\cs$ set
$$\nabla^s_+ \phi=(-\Delta)^{\frac s2}\phi.$$
Then we have the following comments.
\begin{itemize}
\item[\rm (i)]
If $f\in\cs_s'$, then $\nabla^s_+ f$ is defined as a distribution in $\cs'$:
 $$\laz \nabla^s_+f,\,\phi\raz=\laz f,\,\nabla^s_+\phi \raz \qquad \forall\ \ \phi\in\cs.
 $$

\item[\rm (ii)] According to \cite[Proposition~2.4]{Sil}, if $f$ belongs to the weighted-$L^1$ space
$$\mathbb L_s=L_\loc^1\cap\cs_s'=\lf\{f:\, \rn\to  \rr\;\textup{obeys}\; \|f\|_{\mathbb L_s}=\int_\rn \frac{|f(x)|}{1+|x|^{n+s}}\,dx<\infty\r\}$$
and the H\"older space $\mathcal C^{s+\ez_x}$ in a neighborhood of $x\in\rn$ for some $\ez_x\in(0, 1-s]$, then
$\nabla^s_+ f$ is continuous at $x$ and it has the integral expression (cf. \cite[Definition~1.1 \& Lemma~1.4]{Bucur})
\begin{align*}
\nabla^s_+ f(x)= c_{n,s,+} \int_{\mathbb R^n}\frac{f(x)-f(y)}{|x-y|^{n+s}}\,dy\ \ \text{with}\ \
 c_{n,s,+}=\frac{s2^{s-1}\Gamma\big(\frac{n+s}{2}\big)}{\pi^{\frac{n}{2}}\Gamma\big(1-\frac{s}{2}\big)}.
 \end{align*}
 Evidently, {this integral expression}  holds for any $f\in{\cs}$.
 \end{itemize}
\end{definition}

The next lemma shows that $I_s$ is the inverse of $\nabla^s_+$ on $\cs$, and vice versa.

\begin{lemma}\label{lem-x1}
If $(s,\phi,x)\in(0,1)\times\cs\times\rn$, then $$
I_s(-\Delta)^\frac s2 \phi(x)=\phi(x)=(-\Delta)^\frac s2 I_s\phi(x).
$$
\end{lemma}

\begin{proof}
On the one hand, \cite[p.\,117,\, Lemma~1(a)]{St} and Definition \ref{defn1} derive
\begin{align*}
I_s(-\Delta)^\frac s2 \phi(x)
&=c_{n,s}\int_\rn |y|^{s-n} (-\Delta)^\frac s2 \phi(x-y)\, dy\\
&=\int_\rn (2\pi|y|)^{-s} \lf((-\Delta)^\frac s2 \phi(x-\cdot)\r)^\wedge(y)\, dy\\
&=\int_\rn (2\pi|y|)^{-s} e^{2\pi i x\cdot y}\lf((-\Delta)^\frac s2 \phi\r)^\wedge(y)\, dy\\
&=\int_\rn  e^{2\pi i x\cdot y}\hat \phi(y)\, dy\\
&=\phi(x).
\end{align*}

On the other hand, for any $\az\in\mathbb Z^n_+$  we use
$$D^\az I_s\phi(x)= I_sD^\az\phi(x)=c_{n,s}\int_\rn |y|^{s-n} D^\az\phi(x-y)\, dy$$
to get
\begin{align*}
|D^\az I_s\phi(x)|
&\ls \int_\rn |y|^{s-n} (1+|x-y|)^{-(n+1)}\, dy\\
&\ls \int_{|y|<1}  |y|^{s-n} \, dy+\int_{|y|\ge1}   (1+|x-y|)^{-(n+1)}\, dy\\
&\ls 1,
\end{align*}
which implies
$$D^\az I_s\phi\in L^\infty.$$
Accordingly, $I_s\phi\in \mathbb L_s$  and $I_s\phi$ locally satisfies the Lipschitz condition. Now, an application of Definition \ref{defn1}(ii) gives that $(-\Delta)^\frac s2 I_s\phi$ is continuous  on $\rn$.
Furthermore, since
$$I_s\phi\in L^\infty\Rightarrow I_s\phi\in\cs_s',
$$
we have
\begin{align*}
\laz (-\Delta)^\frac s2 I_s\phi, \psi\raz
=\laz  I_s\phi, (-\Delta)^\frac s2\psi\raz
=\laz  \phi, I_s(-\Delta)^\frac s2\psi\raz
=\laz \phi,\psi\raz\ \ \forall\ \ \psi\in\cs,
\end{align*}
where the second identity is from the Fubini theorem and the last identity is due to {the already-proved identification}
$$
I_s(-\Delta)^\frac s2 \psi=\psi\qquad \ \forall\ \ \psi\in\cs.
$$
Accordingly,
$$
(-\Delta)^\frac s2 I_s\phi=\phi \qquad \ \text{in}\ \ \cs'.
$$
But nevertheless,
$$(-\Delta)^\frac s2 I_s\phi\ \ \&\ \ \phi$$ are continuous on $\rn$, so we arrive at
$$(-\Delta)^\frac s2 I_s\phi(x)=\phi(x)\qquad \ \forall\ \ x\in\rn.
$$
\end{proof}

\subsubsection*{About $\nabla^s_- u$}\label{s212} We begin with the following

\begin{definition}\label{defn2}
For $(s,j,\phi)\in(0,1)\times\{1,2,\dots,n\}\times\cs$ let
$$\nabla^s_- \phi=\big(\nabla^s_1\phi,\nabla^s_2\phi,\dots,\nabla^s_n\phi\big),$$
where each $\nabla^s_j \phi$ is defined via the Fourier transform:
$$\lf(\nabla^s_j \phi\r)^\wedge(\xi)=(-2\pi i \xi_j)(2\pi|\xi|)^{s-1}\,\hat \phi(\xi)\qquad \forall\ \ \xi\in\rn.$$
\end{definition}

\begin{lemma}\label{lem-integral-1}
If $(s,\phi,x)\in(0,1)\times\cs\times\rn$, then
\begin{align}\label{nablas}
\nabla^s_- \phi(x){=I_{1-s}\nabla\phi(x)}=c_{n,s,-} \int_{\mathbb R^n}\left(\frac{x-y}{|x-y|}\right)\left(\frac{\phi(x)-\phi(y)}{|x-y|^{n+s}}\right)\,dy
\ \ \text{with}\ \
c_{n,s,-}=\frac{2^{s}\Gamma\big(\frac{n+s+1}{2}\big)}{\pi^{\frac n2}\Gamma\big(\frac{1-s}{2}\big)}.
\end{align}
\end{lemma}

\begin{proof}
With $$c_{n,1-s}=\frac{2^{s-1}\Gamma(\frac{n+s-1}{2})}{\pi^{\frac n2}\Gamma(\frac {1-s}2)},$$
it follows from Definition \ref{defn2} that any $\phi \in\cs$ satisfies
\begin{align*}
\nabla^s_-\phi (x)
&=\lf((2\pi|\xi|)^{s-1}\,(\nabla \phi )^\wedge\r)^\vee(x)\\
&=c_{n,1-s}
\int_\rn \frac{\nabla \phi (x-y)}{|y|^{n-(1-s)}}\, dy\\
&=c_{n,1-s}\lim_{\ez\to 0,\;N\to\infty} \int_{\ez<|y|<N}\frac{\nabla \phi (x-y)}{|y|^{n-(1-s)}}\, dy.
\end{align*}
Note that the second equality in the above formula also implies
$$\nabla^s_- \phi(x)=I_{1-s}\nabla\phi(x)\qquad\forall \ \ x\in\rn.$$
Moreover, the integration by parts formula gives
\begin{align*}
\int_{\ez<|y|<N}\frac{\nabla \phi (x-y)}{|y|^{n-(1-s)}}\, dy&=
\int_{|y|\in\{\ez,N\}}\left(\frac{\phi (x-y)}{|y|^{n-(1-s)}}\right) \vec \nu(y)\, d\ch^{n-1}(y)\\
&\quad
+(s+n)\int_{\ez<|y|<N} \left(\frac {y}{|y|}\right) \left(\frac{\phi (x-y)}{|y|^{n+s}}\right)\, dy,
\end{align*}
where $\vec\nu$ is the outward unit vector on the surface of  the ring $$\{y\in\rn:\, \ez<|y|<N\}$$
and $\ch^{n-1}$ is the $(n-1)$-dimensional Hausdorff measure. An application of
$$
\begin{cases}\vec\nu(y)=-\frac y{|y|}\ \ &\text{when}\ \ |y|=\ez\\
\vec\nu(y)=\frac y{|y|}\ \ &\text{when}\ \  |y|=N
\end{cases}
$$
derives
\begin{align*}
\lf|\int_{|y|=\ez}\left(\frac{\phi (x-y)}{|y|^{n-(1-s)}}\right) \vec \nu(y)\, d\ch^{n-1}(y)\r|
&=\lf|\int_{|y|=\ez}\left(\frac{\phi (x-y)-\phi (x)}{|y|^{n-(1-s)}}\right)\left( \frac y{|y|}\right)\, d\ch^{n-1}(y)\r|
\ls {\ez^{1-s} \big\||\nabla \phi|\big \|_{L^\infty}}
\end{align*}
and
\begin{align*}
\lf|\int_{|y|=N}\left(\frac{\phi (x-y)}{|y|^{n-(1-s)}}\right) \vec \nu(y)\, d\ch^{n-1}(y)\r|{\ls N^{-s} \|\phi \|_{L^\infty}}.
\end{align*}
Consequently, the Lebesgue dominated convergence theorem, along with  $$
\int_{\mathbb R^n}\frac{|\phi (x)-\phi (y)|}{|x-y|^{n+s}}\,dy<\infty,
$$
yields the desired {integral expression in} \eqref{nablas}:
\begin{align*}
\nabla^s_- \phi (x)&=(s+n)c_{n,1-s}\lim_{\ez\to 0,\;N\to\infty} \int_{\ez<|y|<N} \left(\frac {y}{|y|}\right) \left(\frac{\phi (x-y)}{|y|^{n+s}}\right)\, dy\\
&=c_{n,s,-}\lim_{\ez\to 0,\;N\to\infty} \int_{\ez<|x-y|<N}\left( \frac {x-y}{|x-y|}\right)\left(\frac{\phi (x)-\phi (y)}{|x-y|^{n+s}}\right)\, dy\\
&=c_{n,s,-}\int_\rn \left(\frac {x-y}{|x-y|}\right)\left( \frac{\phi (x)-\phi (y)}{|x-y|^{n+s}}\right)\, dy.
\end{align*}
\end{proof}

\begin{lemma}\label{lem-add1}
If $(s,j)\in(0,1)\times\{1,2,\dots,n\}$, then $\nabla^s_j$ maps $\cs$ into $\cs_s$.
\end{lemma}

\begin{proof}
Suppose
$$\az=(\az_1,\dots,\az_n)\in\zz_+^n.$$
Since
$$\phi\in\cs\Rightarrow {D^\alpha} \phi=\partial_{x_1}^{\alpha_1}\cdots\partial_{x_n}^{\az_n}\phi\in\cs,
$$
the Fourier transform gives
$$
D^\alpha \nabla^s_j\phi= \nabla^s_j(D^\alpha\phi).
$$
This, combined with the integral representation of $\nabla^s_j{D^\alpha}\phi$ given in Lemma \ref{lem-integral-1}, yields
\begin{align*}
|D^\alpha \nabla^s_j\phi(x)|
\approx \lf|\int_\rn \left(\frac{x_j-y_j}{|x-y|}\right)\left(\frac{D^\alpha\phi(x)-D^\alpha\phi(y)}{|x-y|^{n+s}}\right)\,dy\r|\qquad\forall \ \ x\in\rn.
\end{align*}
Clearly,
\begin{align*}
\lf|\int_{|x-y|\ge(1+|x|)/2} \left(\frac{x_j-y_j}{|x-y|}\right)\left(\frac{D^\alpha\phi(x)-D^\alpha\phi(y)}{|x-y|^{n+s}}\right)\,dy\r|
& =\lf|\int_{|x-y|\ge(1+|x|)/2} \left(\frac{x_j-y_j}{|x-y|}\right)\left(\frac{D^\alpha\phi(y)}{|x-y|^{n+s}}\right)\,dy\r|\\
&\ls (1+|x|)^{-(n+s)}\|D^\alpha\phi\|_{L^1}.
\end{align*}
Also, the mean-value theorem derives
\begin{align*}
&\lf|\int_{|x-y|<(1+|x|)/2} \left(\frac{x_j-y_j}{|x-y|}\right)\left(\frac{D^\alpha\phi(x)-D^\alpha\phi(y)}{|x-y|^{n+s}}\right)\,dy\r|\\
&\quad\ls\int_{|x-y|<(1+|x|)/2} |x-y|^{1-s-n}\sup_{\tz\in(0,1)}\sup_{|\bz|=|\az|+1}\lf|D^\beta\phi(\tz x+(1-\tz)y)\r|\,dy\\
&\quad\ls (1+|x|)^{-(n+s)}.
\end{align*}
Combining the above two estimates gives
$$|D^\alpha \nabla^s_j\phi(x)|\ls  (1+|x|)^{-(n+s)}\qquad\forall \ \ x\in\rn,$$
and so $$\nabla^s_j\phi\in\cs_s.$$
\end{proof}

 Lemma \ref{lem-add1} can be used to extend the definition of $\nabla^s_-$ to all distributions in $\cs_s'$.

\begin{definition}\label{defn4}
For $(s,f)\in(0,1)\times\cs_s'$ let
$$\nabla^s_- f=(\nabla^s_1f,\nabla^s_2f,\dots,\nabla^s_nf),$$
where $\nabla^s_j \phi$ is defined by
$$\laz \nabla^s_j f, \phi\raz=-\laz f, \nabla^s_j\phi\raz\qquad \forall\ \ \phi\in\cs.$$
\end{definition}

Like Definition \ref{defn1} made for $\nabla^s_+$, we have also the
integral representing of $\nabla^s_-f$ whenever $f\in\mathbb L_s$ has local H\"older regularity.

\begin{lemma}\label{lem-integral}
Let $(s,f,x)\in(0,1)\times\mathbb L_s\times\rn$. If $f$ has the H\"older continuity of order ${s+\ez}$ in a neighborhood $\Omega$ of $x$ for some $\ez\in(0, 1-s]$, then {$\nabla^s_- f$ is continuous at $x$ and }
\begin{align}\label{nabla-s}
\nabla^s_- f(x)=c_{n,s,-} \int_{\mathbb R^n}\left(\frac{x-y}{|x-y|}\right)\left(\frac{f(x)-f(y)}{|x-y|^{n+s}}\right)\,dy
\ \ \text{with}\ \
c_{n,s,-}=\frac{2^{s}\Gamma\big(\frac{n+s+1}{2}\big)}{\pi^{\frac n2}\Gamma\big(\frac{1-s}{2}\big)}.
\end{align}
\end{lemma}

\begin{proof} Without loss of generality, we may assume that $\Omega$ is bounded and naturally $\Omega^\complement=\rn\setminus\Omega$ is unbounded. An application of both
$$\int_\Omega \frac{|f(x)-f(y)|}{|x-y|^{n+s}}\,dy \ls \int_\Omega |x-y|^{\ez-n}\,dy<\infty$$
and
\begin{align*}
\int_{ \Omega^\complement} \frac{|f(x)-f(y)|}{|x-y|^{n+s}}\,dy
\ls \int_{ \Omega^\complement} \frac{|f(x)|+|f(y)|}{(\dist(x, \Omega^\complement)+|y|)^{n+s}}\,dy
\ls |f(x)|+\|f\|_{\mathbb L_s}<\infty
\end{align*}
derives
$$\int_{\mathbb R^n}\frac{|f(x)-f(y)|}{|x-y|^{n+s}}\,dy<\infty,
$$
and hence the integral in the right-hand-side of \eqref{nabla-s} converges absolutely.

To show \eqref{nabla-s}, we take an arbitrary open set $\Omega_0\ni x$ compactly contained in $\Omega$. According to the proof of \cite[Proposition~2.4]{Sil}, there
exists a sequence $\{f_k\}_{k\in\nn} \subset\cs$ uniformly bounded in $C^{s+\ez}(\Omega)$, converging
uniformly to $f$ in $\Omega_0$ and also converging to $f$ in the norm of $\mathbb L_s$.
For any $k\in\nn$, since $f_k\in\cs$, we utilize Lemma \ref{lem-integral-1} to write
\begin{align*}
\nabla^s_- f_k(x)=c_{n,s,-}\int_\rn \left(\frac {x-y}{|x-y|}\right)\left( \frac{f_k(x)-f_k(y)}{|x-y|^{n+s}}\right)\, dy\qquad\forall\ \ x\in\rn.
\end{align*}
From the uniform
bound on the $C^{s+\ez}$-norm of $f_k$ in $\Omega_0$ it follows that
$$\int_{\mathbb R^n}\left(\frac{x-y}{|x-y|}\right)\left(\frac{f_k(x)-f_k(y)}{|x-y|^{n+s}}\right)\,dy
\to \int_{\mathbb R^n}\left(\frac{x-y}{|x-y|}\right)\left(\frac{f(x)-f(y)}{|x-y|^{n+s}}\right)\,dy$$
uniformly in $\Omega_0$ as $k\to\infty$.
Since $\{f_k\}_{k\in\nn}$ converges to $f$ in the norm of $\mathbb L_s$, it follows easily that
$$\nabla^s_- f_k\to \nabla^s_-f\ \ \ \textup{in}\ \ \cs_0'.
$$
Accordingly, $\nabla^s_-f(x)$ must
coincide with $$c_{n,s,-}\int_{\mathbb R^n}\left(\frac{x-y}{|x-y|}\right)\left(\frac{f(x)-f(y)}{|x-y|^{n+s}}\right)\,dy$$ in $\Omega_0$ by the uniqueness of the limits. So, \eqref{nabla-s} holds.
\end{proof}

Below is more information on $\nabla^s_-$.

\begin{lemma}\label{lem-add3}
Let $s\in(0,1)$.
\begin{enumerate}
\item[\rm (i)] If $\phi\in\cs_\infty$, then it holds pointwisely on $\rn$ that
\begin{align*}
\nabla^s_-\phi
=I_{1-s}\nabla\phi
=\nabla I_{1-s}\phi
=\vec{R}(-\Delta)^\frac s2 \phi
=(-\Delta)^\frac s2 \vec{R} \phi.
\end{align*}

\item[\rm (ii)]
If $\phi\in\cs_s'$, then $\nabla^s_-\phi
=I_{1-s}\nabla\phi$ in $\cs'/\mathcal P$.

\item[\rm (iii)] If  $\phi\in \cs$, then all  identities in (i) hold almost everywhere on $\rn$.
\end{enumerate}

\end{lemma}

\begin{proof}
(i) Via the Fourier transform, we see that
$$
I_{1-s},\ \ (-\Delta)^\frac s2\ \ \&\ \ R_{1\le j\le n}
$$
map $\cs_\infty$ into $\cs_\infty$. Then, taking the inverse Fourier transform verifies the assertion in (i).

(ii) Let
$$(\phi,j,\psi)\in\cs_s'\times\{1,2,\dots,n\}\times\cs_\infty.
$$
Then by the just-checked (i) and Definition \ref{defn4} we have
\begin{align*}
 \laz\nabla^s_j\phi,\psi\raz
=-\laz \phi, \nabla^s_j\psi\raz
=
-\laz \phi, \partial_{x_j} I_{1-s}\psi\raz
\end{align*}
Further, since
$$\phi\in\cs_s'\Rightarrow\phi\in\cs'\Rightarrow\partial_{x_j}\phi\in\cs'\subset\cs'/\mathcal P,
$$
this implication, along with
the fact that $I_{1-s}$ maps $\cs'/\mathcal P$ onto $\cs'/\mathcal P$, derives
$$
-\laz \phi, \partial_{x_j} I_{1-s}\psi\raz=\laz \partial_{x_j}\phi,  I_{1-s}\psi\raz
=\laz I_{1-s}\partial_{x_j}\phi,  \psi\raz,
$$
namely,
$$\nabla^s_j\phi= I_{1-s}\partial_{x_j}\phi\qquad \text{in}\ \ \cs'/\mathcal P.$$

(iii)  Observe that $\cs_\infty$ is dense in $L^p$ whenever $p\in(1,\infty)$.
Indeed, this follows easily from the fact that
 the Calder\'on reproducing formula of an $L^p$-function $$
f=\int_0^\infty \varphi_t\ast\psi_t\ast f\,\frac{dt}{t}
$$
holds in $L^p$ (cf. \cite[p.8, Theorem~(1.2)]{FJW} for $p=2$ and  \cite{Saeki} for general $p$),
with
$\psi,\varphi\in\cs_\infty$ satisfying
$$
\begin{cases}
\supp\hat\varphi,\supp\hat\psi\subset \{\xi\in\rn:\, 1/4\le |\xi|\le 4\}\\
\min\big\{|\hat\varphi(\xi)|, |\hat\psi(\xi)|\big\}>c\ \text{on}\  \{\xi\in\rn:\, 1/2\le |\xi|\le 2\}\\
\int_0^\infty \hat\varphi(t\xi)\hat\psi(t\xi)\, \frac{dt}{t}=1\ \text{for}\; \xi\neq0.
\end{cases}
$$
Thus, if $\phi\in\cs$, then a  discussion similar to (ii) yields that the identity in (i) holds in $\cs'/\mathcal P$.
Moreover,  by the density of $\cs_\infty$ in $L^2$ and the duality equality
$$
\|u\|_{L^2}
=\sup\lf\{|\laz u, \phi\raz|:\, \phi\in\cs_\infty, \|\phi\|_{L^2}\le 1\r\},
$$
we obtain that the identity in (i) holds in $L^2$ and hence almost everywhere on $\rn$.
\end{proof}

\subsection{Dense subspaces of $H^{s,1}\ \&\ H^{s,1}_+$}\label{s22}
Note that $\cs_\infty$ is dense in $H^1$. However, instead of $\cs_\infty$ we may consider the following larger space
\begin{align}\label{Ss0}
\cs_{0<s<\infty,0}=\lf\{f\in C^\infty:\, \int_\rn f(x)\, dx=0\; \&\;
\sup_{x\in\rn}(1+|x|)^{n+s}|f(x)|<\infty\r\}.
\end{align}

\subsubsection*{A dense subspace of $H^1$}\label{s221} As showed in the coming-up-next Lemma \ref{lem-H1} whose argument relies on the radial maximal function characterization of the Hardy space $H^1$ (cf. \cite{St2}), the class defined by \eqref{Ss0} is a dense subspace of $H^1$.
To see this, recall that if
\begin{align}\label{phi}
\begin{cases}
0\le\phi\in \cs\\
\int_\rn \phi(x)\, dx=1\\
\phi_t(x)=t^{-n}\phi(t^{-1}x)\;\forall\, (t,x)\in(0,\infty)\times\rn,
\end{cases}
\end{align}
then
$$
H^1=\lf\{
f\in\cs':\  f^+=\sup_{t\in(0,\infty)} |\phi_t\ast f|\in L^1
\r\}\ \ \text{with}\ \
\|f\|_{H^1}\approx \| f^+\|_{L^1}.
$$
We are led to discover the following density for $H^1$.

\begin{lemma}\label{lem-H1}
Let $s\in(0,\infty)$. Then any locally integrable function $f$ on $\rn$ with
$$
\int_\rn f(x)\, dx=0\ \ \&\ \
\sup_{x\in\rn}(1+|x|)^{n+s}|f(x)|<\infty
$$
belongs to the Hardy space $H^1$. Consequently, $\cs_{s,0}$ is dense in $H^1$. Moreover,
$$\lf\{(-\Delta)^\frac s2\phi:\, \phi\in\cs\r\}\subset \cs_{s,0}\subset H^1\qquad \forall\ s\in(0,1).$$

\end{lemma}

\begin{proof}
Let $\phi$ and $\{\phi_t\}_{t\in(0,\infty)}$ be as in \eqref{phi}. By the radial maximal function characterization of $H^1$, we only need to show that
\begin{align}\label{eq-aim}
|\phi_t\ast f(x)|\ls (1+|x|)^{-(n+\ez)} \qquad\forall\ \ (t,x)\in(0,\infty)\times\rn
\end{align}
holds for some $\ez\in(0,s)$. Indeed,
$$
\eqref{eq-aim}\Rightarrow f^+(x)\ls (1+|x|)^{-(n+\ez)}\ \ \forall\ \  x\in\rn\Rightarrow f^+\in L^1.
$$
However, \eqref{eq-aim} is verified by handling two {situations: $|x|\le 1$ and $|x|>1$}.

If $|x|\le 1$, then
$$
|\phi_t\ast f(x)|\ls \|f\|_{L^\infty}\int_\rn \phi_t(x-y)\, dy\ls 1\approx (1+|x|)^{-(n+\ez)}.
$$

If $|x|\ge 1$, then by the conditions of $f$ we write
\begin{align*}
|\phi_t\ast f(x)|
=\lf|\int_\rn (\phi_t(x-y)-\phi_t(x))f(y)\,dy\r|
\ls \int_\rn \frac{\lf|\phi_t(x-y)-\phi_t(x)\r|}{(1+|y|)^{(n+s)}}\, dy.
\end{align*}
On the one hand, the mean value theorem gives
\begin{align*}
&\int_{|y|<|x|/2} \lf|\phi_t(x-y)-\phi_t(x)\r| (1+|y|)^{-(n+s)}\, dy\\
&\quad\le \int_{|y|<|x|/2} t^{-n-1} \sup_{\theta\in(0,1)} \lf|\nabla\phi\lf(t^{-1}(x-\tz y)\r)\r|(1+|y|)^{-(n+s)}\, dy\\
&\quad\ls \int_{|y|<|x|/2} t^{-n-1} (1+t^{-1}|x|)^{-(n+1)}(1+|y|)^{-(n+s)}\, dy\\
&\quad\ls |x|^{-(n+1)}\int_{|y|<|x|/2}(1+|y|)^{-(n+s)}\, dy\\
&\quad\ls |x|^{-(n+1)}.
\end{align*}
On the other hand,
\begin{align*}
&\int_{|y|\ge |x|/2} \lf|\phi_t(x-y)-\phi_t(x)\r| (1+|y|)^{-(n+s)}\, dy\\
&\quad\ls (1+|x|)^{-(n+s)} \int_{|y|\ge |x|/2} \lf|\phi_t(x-y)\r| \, dy
+(1+|x|)^{-\ez} \int_{|y|\ge |x|/2} \frac{t^{-n}(1+t^{-1}|x|)^{-n}}{(1+|y|)^{n+s-\ez}}\, dy\\
&\quad\ls (1+|x|)^{-(n+s)}+|x|^{-(n+\ez)}\int_{|y|\ge |x|/2}(1+|y|)^{-(n+s-\ez)}\, dy\\
&\quad \ls (1+|x|)^{-(n+s)}+|x|^{-(n+\ez)}.
\end{align*}
Via combining the last three formulae we obtain
$$
|\phi_t\ast f(x)|\ls |x|^{-(n+\ez)}\approx (1+|x|)^{-(n+\ez)}\qquad  \forall\ \ |x|\ge 1,
$$
thereby reaching \eqref{eq-aim}.

{The remaining part of Lemma \ref{lem-H1} is obvious. }
\end{proof}

\subsubsection*{The first and second dense subspaces of $H^{s,1}$}\label{s222} Lemma \ref{lem-H1} produces the following property.
\begin{proposition}\label{prop1}
Let $s\in (0,1)$. Then

\begin{enumerate}
\item[\rm (i)] $H^{s,1}\cap \cs_\infty= I_s(\cs_\infty\cap H^1).$

\item[\rm (ii)] $I_s(H^1)\subset H^{s,1}$.

\item[\rm (iii)] For any $f\in H^{s,1}$ there exists $g\in H^1$ such that $f=I_s g$ in $\cs'/\mathcal P$.
\end{enumerate}

\end{proposition}

\begin{proof} (i)
For any $\phi\in\cs_\infty$, by the invariant property of $\cs_\infty$ under the action of $I_s$ or $(-\Delta)^\frac s2$, we get
$$
\phi\in H^{s,1}\Leftrightarrow(-\Delta)^\frac s2 \phi\in H^1,
$$
as desired.

(ii) If $f\in I_s(H^1)$, then
$$f=I_s g\ \ \text{for some}\ \ g\in H^1,
$$
and hence
 $$
 \eqref{e1}\Rightarrow f\in L^{\frac n{n-s}}.
 $$
 Of course, any function in $L^{\frac n{n-s}}$ belongs to  $ \cs_s'$. Accordingly, for any $\phi\in\cs$, by Lemma \ref{lem-x1} we have
 \begin{align*}
 \laz (-\Delta)^\frac s2 f, \phi\raz
 = \laz  f, (-\Delta)^\frac s2\phi\raz
 = \laz  I_s g, (-\Delta)^\frac s2\phi\raz
 =\laz   g, I_s(-\Delta)^\frac s2\phi\raz
 =\laz g, \phi\raz,
 \end{align*}
 where {in the penultimate equality} the Fubini theorem has been applied due to the implication
 that if $$g\in H^1\subset L^1\ \ \&\ \ (-\Delta)^\frac s2\phi\in \cs_s$$ then \begin{align*}
 \int_\rn \int_\rn |x-y|^{s-n} |g(y)||(-\Delta)^\frac s2\phi(x)|\, dx\, dy
 \ls \int_\rn \int_\rn \frac{|x-y|^{s-n} |g(y)| }{1+|x|^{n+s}}\, dx\, dy
 \ls \|g\|_{L^1}.
 \end{align*}
Therefore, we obtain
 $$
 (-\Delta)^\frac s2 f=g \qquad \text{in} \ \ \cs'.
 $$
Since $g$ belongs to $H^1$, so does $(-\Delta)^\frac s2 f$. This proves $$I_s(H^1)\subset H^{s,1}.$$

(iii) Recall that both $I_s$ and $(-\Delta)^\frac s2$ are one-to-one maps from $\cs'/\mathcal P$ to $\cs'/\mathcal P$. Thus, we have
 $$f\in\cs_s'\subset\cs'/\mathcal P\qquad \forall\ \ f\in H^{s,1},
$$ thereby getting
$$f=I_s\big((-\Delta)^\frac s2 f\big)\ \ \text{in}\ \  \cs'/\mathcal P$$
and so
$$f=I_s g\ \ \text{in}\ \ \cs'/\mathcal P\ \ \text{with}\ \  g=(-\Delta)^\frac s2 f\in H^1.
$$
\end{proof}

Next, we have the following density result.

\begin{proposition}\label{prop-dense1}
If $s\in(0,1)$, then $$\cs_\infty\subset\cs\subset H^{s,1}\subset H^{s,1}_\pm.
$$ Moreover, both  $\cs_\infty$ and $\cs$ are dense in $H^{s,1}$.
\end{proposition}

\begin{proof}
For any $u\in\cs$, we {use Lemma \ref{lem-H1}} to derive
$$(-\Delta)^\frac s2 u\in H^1\ \ \text{i.e.}\ \ u\in H^{s,1}.
$$
This proves $\cs\subset H^{s,1}$; the other inclusions are obvious.

It remains to show the density of $\cs_\infty$ in $H^{s,1}$. If $f\in H^{s,1}$, then
$$
f\in \cs_s'\ \ \&\ \ (-\Delta)^\frac s2 f\in H^1.
$$
Due to the density of $\cs_\infty$ in $H^1$,
$$\exists\ \{g_j\}_{j\in\nn}\subset \cs_\infty\ \ \text{such that}\ \
\lim_{j\to\infty}\|g_j-(-\Delta)^\frac s2 f\|_{H^1}\to 0.$$
For any $j\in\nn$, let $$f_j=I_sg_j,$$
which actually belongs to $\cs_\infty$ in terms of the Fourier transform.
Noticing that
$$g_j=(-\Delta)^\frac s2 f_j,
$$
 we have
$$
\lim_{j\to\infty}[f_j-f]_{H^{s,1}}=\lim_{j\to\infty}\|(-\Delta)^\frac s2 (f_j- f)\|_{H^1}=\lim_{j\to\infty}\|g_j-(-\Delta)^\frac s2 f\|_{H^1}=0.
$$
Thus, $f\in H^{s,1}$  can be approximated by the $\cs_\infty$-functions $\{f_j\}_{j\in\nn}$.
\end{proof}

\subsubsection*{A dense subspace of $H^{s,1}_+$}\label{s223}
It is difficult to determine the density of  $\cs$ in $H^{s,1}_\pm$. However, we have

\begin{proposition}\label{prop-dense3}
If $s\in (0,1)$, then $I_s(\cs)$ is a dense subspace of $H^{s,1}_+$ but $$I_s(\cs)\not\subset H^{s,1}_-.$$
\end{proposition}

\begin{proof} On the one hand, if $f\in I_s(\cs)$ then
	$$\exists\ \phi\in\mathcal{S}\ \ \text{such that}\ \ f=I_s\phi,
	$$
but Lemma \ref{lem-x1} implies
$$(-\Delta)^\frac s2 f = (-\Delta)^\frac s2I_s\phi=\phi\in\cs\subset L^1,$$
that is,
$
f\in H^{s,1}_+.
$
To show the density of $I_s(\cs)$ in $H^{s,1}_+$, given any $f\in H^{s,1}_+$ we utilize
$$(-\Delta)^\frac s2 f\in L^1$$ and the density of  $\cs$  in $L^1$ to find a sequence
 $$\{g_j\}_{j\in\nn}\subset\cs$$ such that
$$\lim_{j\to\infty}\|g_j-(-\Delta)^\frac s2 f\|_{L^1}=0.$$
Upon defining $$f_j=I_s g_j\in I_s(\cs)$$ and using Lemma \ref{lem-x1}, we gain the representation
$$g_j=(-\Delta)^\frac s2 f_j$$
and the desired convergence
$$
\lim_{j\to\infty}[f_j-f]_{H^{s,1}_+}
=\lim_{j\to\infty}\|(-\Delta)^\frac s2 (f_j-f)\|_{L^1}
=\lim_{j\to\infty}\|g_j-(-\Delta)^\frac s2 f\|_{L^1}=0.
$$
In other words, $I_s(\cs)$ is a dense subspace of $H^{s,1}_+$.

On the other hand, $I_s(\cs)$ is not a subspace of  $H^{s,1}_-$ -
otherwise - if $I_s(\cs)\subset H^{s,1}_-$, then this, along with $I_s(\cs)\subset H^{s,1}_+$, would imply
$I_s(\cs)\subset H^{s,1}$ and hence $\cs\subset H^1$ which is impossible.
\end{proof}

\subsubsection*{The third dense subspace of $H^{s,1}$}\label{s224} In addition to Proposition \ref{prop-dense1}, we obtain
\begin{proposition}\label{prop-dense2}
If $s\in (0,1)$, then  $$\cd_0=\Bigg\{f\in C_c^\infty:\, \int_\rn f(x)\, dx=0\Bigg\}$$ is a dense subspace of $H^{s,1}$.
\end{proposition}

\begin{proof} Proposition \ref{prop-dense1} implies
$$\cd_0\subset\cs\subset H^{s,1}.
$$
So, it suffices to show the density of $\cd_0$ in $H^{s,1}$.
Let $f\in H^{s,1}$.
Based on Proposition \ref{prop1}(iii),
$$\exists\ g\in H^1\ \ \text{such that}\ \  f=I_sg\ \ \text{in}\ \  \cs'/\mathcal P.
$$
Note that $H^1$ is nothing but the homogeneous Triebel-Lizorkin space $\dot F^0_{1,2}$. So, the lifting property of $I_s$ on the Triebel-Lizorkin spaces (cf. \cite[p.\,242]{Tr83}) shows
$$I_s(H^1)=I_s(\dot F^0_{1,2})=\dot F^s_{1,2}.$$
Therefore,
$$\exists\ \wz f\in \dot F^s_{1,2}\ \ \text{such that}\ \
f=I_sg=\wz f\ \ \text{in}\ \ \cs'/\mathcal P.
$$

Recall that \cite[Theorem~1]{HPW} yields that any element in $\dot F^s_{1,2}$ can be written as the
linear combinations of $\dot F^s_{1,2}$-atoms, just as the  atomic decomposition of the Hardy space $H^1$.
To be precise, since $\wz f\in \dot F^s_{1,2}$,  it follows that
$$\wz f=\sum_{j\in\nn} \lz_j a_j \ \ \textup{in}\ \ \cs'/\mathcal P,$$
where
$$\|\wz f\,\|_{\dot F^s_{1,2}}\approx \sum_{j\in\nn}|\lz_j|
$$
and, based on the remark after \cite[Definition~(1.6)]{HPW}, every $a_j$ is a locally integrable function on $\rn$ with the following three properties:
\begin{enumerate}
\item[\rm (i)] $a_j$ is supported on a ball $B_j$;

\item[\rm (ii)] $\|a_j\|_{\dot F^s_{2,2}}\le |B_j|^{-\frac 12}$;

\item[\rm (iii)] $\int_\rn a_j\,dx=0$.
\end{enumerate}
Again using the lifting property of $(-\Delta)^\frac s2$ (cf. \cite[p.\,242]{Tr83}) gives
$$\|a_j\|_{\dot F^s_{2,2}}=\|(-\Delta)^\frac s2a_j\|_{\dot F^0_{2,2}}.$$
By \cite[p.\,242, (2)]{Tr83}, any element in $\dot F^0_{2,2}$ coincides with a function in $L^2$ in the sense of $\cs'/\mathcal P$.
Thus, we know from $(-\Delta)^\frac s2a_j\in \dot F^0_{2,2}$ that
$(-\Delta)^\frac s2a_j$ coincides with some $L^2$-function, denoted by $\wz{(-\Delta)^\frac s2 a_j}$, in the sense of $\cs'/\mathcal P$.
 So,
by the density of $\cs_\infty$ in $L^2$ (cf. the proof of Lemma \ref{lem-add3}(iii)) and the duality we get
\begin{align}\label{eq-z1}
\|(-\Delta)^\frac s2a_j\|_{L^2}
&=\sup\lf\{|\laz (-\Delta)^\frac s2a_j, \phi\raz|:\, \phi\in\cs_\infty, \|\phi\|_{L^2}\le 1\r\}\notag\\
&=\sup\lf\{|\laz \wz{(-\Delta)^\frac s2a_j}, \phi\raz|:\, \phi\in\cs_\infty, \|\phi\|_{L^2}\le 1\r\}\notag\\
&=\|\wz{(-\Delta)^\frac s2a_j}\|_{L^2}\\
&\approx  \|(-\Delta)^\frac s2a_j\|_{\dot F^0_{2,2}}\notag\\
& \approx  \|a_j\|_{\dot F^s_{2,2}}.\notag
\end{align}

Let $\psi\in C_c^\infty$ satisfy
$$\int_\rn\psi\, dx=1\ \ \&\ \ \supp\psi\subset B(0,1).
$$
For any $\ez\in(0,\infty)$, define
$$\psi_\ez(\cdot)=\ez^{-n}\psi(\ez^{-1}\cdot).$$
Fix an arbitrary small number $\eta\in(0,\infty)$. For any $j\in\nn$, an application of $(-\Delta)^\frac s2a_j\in L^2$ produces a sufficiently small $\ez_j\in(0,\infty)$ such that
$$\supp\psi_{\ez_j}\ast a_j\subset 2B_j\ \ \&\ \
\|(-\Delta)^\frac s2 a_j-\psi_{\ez_j}\ast((-\Delta)^\frac s2 a_j)\|_{L^2}<\eta|2B_j|^{-\frac 12}.
$$
By \eqref{eq-z1}, the last inequality is equivalent to that
$$
\|a_j-\psi_{\ez_j}\ast a_j\|_{\dot F^s_{2,2}}<\eta|2B_j|^{-\frac 12}.
$$
Choose $N$ large enough such that
$$\sum_{j=N+1}^\infty |\lz_j|<\eta,
$$
and define
$$f_{\ez,N}=\sum_{j=1}^N \lz_j \psi_{\ez_j}\ast a_j.$$
Evidently,
$$\psi_{\ez_j}\ast a_j\in\cd_0\ \ \&\ \
f_{\ez, N}\in\cd_0.
$$
By the argument in \cite[p.\,239]{HPW}, we know that any $\dot F^s_{1,2}$-atom $a_j$ satisfies
$\|a_j\|_{\dot F^s_{1,2}}\ls 1.$
The choice of $\ez_j$ implies that $$\eta^{-1} (a_j-\psi_{\ez_j}\ast a_j)$$ is also an $\dot F^s_{1,2}$-atom, thereby yielding
$$
\|a_j-\psi_{\ez_j}\ast a_j\|_{\dot F^s_{1,2}}\ls \eta.
$$
Upon recalling $$f=\wz f\ \ \text{in}\ \ \cs'/\mathcal P,$$ we obtain
\begin{align*}
\|f_{\ez,N}- f\|_{\dot F^s_{1,2}}
&=\|f_{\ez,N}-\wz f\|_{\dot F^s_{1,2}}\\
&\ls \sum_{j=1}^N |\lz_j|\|\psi_{\ez_j}\ast a_j-a_j\|_{\dot F^s_{1,2}}
+\sum_{j=N+1}^\infty |\lz_j|\|a_j\|_{\dot F^s_{1,2}}\\
&\ls \eta \sum_{j=1}^N |\lz_j|+\sum_{j=N+1}^\infty |\lz_j|\\
&\ls \eta.
\end{align*}

Finally, using the lifting property of $I_s$ {(cf. \cite[p.\,242]{Tr83})} yields
\begin{align*}
[f_{\ez,N}-f]_{H^{s,1}}
=\|(-\Delta)^\frac s2 (f_{\ez,N}-f)\|_{H^1}
\approx \|f_{\ez,N}-\wz f\|_{\dot F^s_{1,2}}\ls \eta.
\end{align*}
Due to the arbitrariness of $\eta$, we obtain that $f\in H^{s,1}$ can be approximated by functions in $\cd_0$.
\end{proof}

\section{Tracing laws for $H^{s,1}\, \&\, H^{s,1}_\pm$}\label{s3}

\subsection{Strong/weak estimates for $\text{Cap}_{X\in\{H^{s,1},H^{s,1}_\pm\} }$}\label{s3.1}

This section is devoted to a measure-theoretic study of the capacity living on $X\in\{W^{s,1},H^{s,1},H^{s,1}_\pm\}$.

\subsubsection*{Capacitary concepts}\label{s311} For $\az\in(0,n)$, denote by $\Lambda^{\alpha}_{(\infty)}$ the $\alpha$-dimensional Hausdorff capacity:
$$
\Lambda^{\alpha}_{(\infty)}(E)=\inf\lf\{\sum_i r_i^\alpha:\, E\subset\bigcup_i B(x_i, r_i),\; (x_i, r_i)\in\rn\times(0,\infty)\r\}
$$
for any set $E\subset\rn$ which is covered by a sequence of balls $$B(x_i,r_i)=\big\{x\in\rn: |x-x_i|<r_i\big\}.$$
Classically, $\Lambda^\alpha_{(\infty)}(\cdot)$ is a monotone, countably subadditive set function on the class of all subsets of $\rn$, and enjoys $\Lambda^\alpha_{(\infty)}(\emptyset)=0$.

\begin{definition}
	\label{d31}
For $s\in (0,1)$ and any compact set $K\subset\rn$ define (cf. \cite{Xadv, Adams})
\begin{equation}\label{e31}
\text{Cap}_{X}(K)=\begin{cases} \inf\big\{[u]_{X}:\ u\in C_c^\infty\ \&\ u\ge 1_K\big\} &\text{as}\ \ X=W^{s,1}\\
\inf\big\{[u]_{X}:\ u\in\cs\ \&\ u\ge 1\ \text{on}\ K\big\}\ &\text{as}\ \ X\in\big\{H^{s,1},H^{s,1}_\pm\big\}.
\end{cases}
\end{equation}
Furthermore, $\text{Cap}_{X}(\cdot)$ is extendable from compact sets to general sets as seen below.
\begin{itemize}
	\item[\rm (i)] If $O\subset\rn$ is open, then
\begin{align*}
\text{Cap}_{X}(O)= \sup_{K\;\textup{compact},\, K\subset O} \text{Cap}_{X}(K)
\end{align*}
\item[\rm (ii)] For an arbitrary set $E\subset\rn$ set
\begin{align*}
\text{Cap}_{X}(E)= \inf_{O\;\textup{open},\, O\supset E} \text{Cap}_{X}(O).
\end{align*}
\end{itemize}
Thus, the definition of $\text{Cap}_{X}$ on any compact/open set is consistent (cf. \cite[Lemma~3.2.4]{LXYY}).
\end{definition}

\begin{lemma}\label{p31}
Let $$
\begin{cases}s\in (0,1)\\
(x,r)\in\mathbb R^n\times(0,\infty)\\
B(x,r)=\{y\in\rn: |y-x|<r\}\\
X\in\big\{H^{s,1},H^{s,1}_+,H^{s,1}_-\big\}.
\end{cases}
$$
Then

\begin{enumerate}	
	
\item[(i)] $\text{Cap}_{X}(\emptyset)=0\ \ \&\ \ \text{Cap}_{X}(B(x,r))=r^{n-s}\text{Cap}_{X}(B(0,1))$.

\item[(ii)] $\text{Cap}_X(E_1)\le\text{Cap}_X(E_2)$ whenever $E_1\subset E_2\subset\rn$.

\item[(iii)]
$$
\max\Big\{\text{Cap}_{H^{s,1}_+}(\cdot),\text{Cap}_{H^{s,1}_-}(\cdot)\Big\}\le
\text{Cap}_{H^{s,1}}(\cdot)\approx \Lambda_{(\infty)}^{n-s}(\cdot)\approx\text{Cap}_{W^{s,1}}(\cdot).
$$

\item[(iv)] $\text{Cap}_{W^{s,1}}(\cdot)$ is countably subadditive, but $\text{Cap}_{H^{s,1}}(\cdot)$ and $\text{Cap}_{H^{s,1}_\pm}(\cdot)$ {may not be} countably subadditive.
\end{enumerate}
\end{lemma}

\begin{proof}  Both (i) and (ii) follow from \eqref{e31}.

(iii) First, according to Definition \ref{d31}, we only need to consider these capacities on compact sets.
For any $u\in\cs$, by \eqref{e2a},  we get
$$[u]_{H^{s,1}_\pm}\le [u]_{H^{s,1}}\ls [u]_{W^{s,1}} \Rightarrow\text{Cap}_{H^{s,1}_\pm}(\cdot)\le\text{Cap}_{H^{s,1}}(\cdot)\ls \text{Cap}_{W^{s,1}}(\cdot).
$$
Noting that
$$\text{Cap}_{W^{s,1}}(\cdot)\approx \Lambda^{n-s}_{(\infty)}(\cdot)$$ is given in \cite[Theorem 2.1]{PS} and \cite[(2.1)]{Xadv},
we are left to verify
\begin{align}\label{eq-x5}
\Lambda^{n-s}_{(\infty)}(\cdot)\ls \text{Cap}_{H^{s,1}}(\cdot).
\end{align}

 According to \cite[Proposition~3]{Adams},  for any compact set $K$ in $\rn$, the capacity
$$
R_s(K)=\inf\Big\{\|f\|_{H^1}:\,\, f\in\cs_\infty\ \&\ I_{s}f\ge 1\; \textup{on}\; K\Big\}
$$
satisfies
\begin{align*}
\Lambda^{n-s}_{(\infty)}(K) \approx R_s(K).
\end{align*}
By Lemma \ref{lem-H1}, we have $$\cs_\infty\subset \mathcal S_{s,0}\subset H^1$$
and  the density of  $\cs_{s,0}$  in $H^1$.
Meanwhile, for any $(f,x)\in \cs_{s,0}\times\rn$, it is obvious that
$I_sf(x)$ is well defined and $I_sf$ is continuous on $\rn$. Thus, instead of using $\cs_\infty$,  we have
{$$
R_s(K)=
\inf\Big\{\|f\|_{H^1}:\,\, f\in\cs_{s,0}\ \&\ I_{s}f\ge 1\; \textup{on}\; K\Big\}.
$$}
For any $ u\in\cs $ satisfying $u\ge 1$ on $K$ let $$f_\ast=(-\Delta)^\frac s2 u,$$ which belongs to $\cs_{s,0}$ in terms of Lemma \ref{lem-H1}.
Then, by  Lemma \ref{lem-x1} we have
$$I_{s}f_\ast=u\ge 1\ \ \text{on}\ \ K,
$$
thereby achieving
\begin{align*}
R_s(K)
&\le \|f_\ast\|_{H^1}
=\|(-\Delta)^\frac s2 u\|_{H^1}=[u]_{H^{s,1}}.
\end{align*}
Taking the infimum over all such $u\in\cs$ satisfying $u\ge 1$ on $K$ yields
\begin{align*}
R_s(K)
\le \inf\Big\{[u]_{H^{s,1}}:\,\, \cs\ni u\ge 1\; \textup{on}\; K\Big\}=\text{Cap}_{H^{s,1}}(K).
\end{align*}
Thus,
\begin{align*}
\Lambda^{n-s}_{(\infty)}(K)\ls \text{Cap}_{H^{s,1}}(K).
\end{align*}
This proves \eqref{eq-x5}.

(iv)  The countable subadditivity of $\text{Cap}_{W^{s,1}}(\cdot)$ follows from \cite[Theorem 1(iii)]{Xcr}. Since the test functions used in
$$
\text{Cap}_X(\cdot)\ \ \text{for}\ \ X\in\big\{H^{s,1},H^{s,1}_\pm\big\}
$$
are not assumed to be nonnegative, the capacities under consideration may not be countably subadditive as mentioned in \cite{Adams}.
\end{proof}

\subsubsection*{Strong estimates for $\text{Cap}_{X\in\{H^{s,1},H^{s,1}_-(n>1)\}}$}\label{s3.2}

First of all, an application of Proposition \ref{p31}(iii) and \cite[Theorem 1.1]{Xadv} or \cite[Theorem~1.3]{PS} gives the following strong inequality for $\text{Cap}_{W^{s,1}}$ (cf. \cite[Theorem 2.2]{Xadv}):
\begin{equation}
\label{e34}
\int_0^\infty \text{Cap}_{W^{s,1}}\big(\{x\in\rn:\, |u(x)|>t\}\big)\,dt\ls [u]_{W^{s,1}} \qquad \ \forall\ \ u\in C_c^\infty.
\end{equation}
Next, we are led by \eqref{e34} to get the strong inequality for $\text{Cap}_{H^{s,1}}$ as seen below.

\begin{theorem}\label{thm3.1}
If $s\in (0,1)$, then
$$
\int_0^\infty \text{Cap}_{H^{s,1}}\big(\{x\in\rn:\, |u(x)|>t\}\big)\,dt\lesssim [u]_{H^{s,1}} \qquad \ \forall\ \ u\in \cs.
$$
\end{theorem}

\begin{proof}
Note that  Proposition \ref{p31}(iii) implies
$$\text{Cap}_{H^{s,1}}(\cdot)\ls \text{Cap}_{W^{s,1}}(\cdot)\approx \Lambda_{(\infty)}^{n-s}(\cdot)$$
and \cite[Proposition~5]{Adams} gives that
\begin{align*}
 \int_0^\infty  \Lambda_{(\infty)}^{n-s}\big(\{x\in\rn:\, |I_s f(x)|>t\}\big)\,dt \ls \|f\|_{H^1} \qquad \ \forall\ \ f\in\mathcal{S}_{s,0}.
\end{align*}
In particular, given  $u\in \cs$, we can take $$f=\nabla^s_+u=(-\Delta)^\frac s2 u,
$$
which belongs to $\cs_{s,0}$ via Lemma \ref{lem-H1}. Noting that Lemmas \ref{lem-x1} and \ref{lem-add3}(iii) imply
$$
\begin{cases} u=I_s f\\
 \nabla^s_-u= \vec{R}(-\Delta)^\frac s2 u=\vec{R}f\; \text{almost everywhere on}\; \rn\\
[u]_{H^{s,1}}=\|\nabla^s_+u\|_{L^1}+\|\nabla^s_- u\|_{L^1}=\|f\|_{L^1}+\|\vec{R}f\|_{L^1}=\|f\|_{H^1},
\end{cases}
$$
we obtain
\begin{align*}
\int_0^\infty \text{Cap}_{H^{s,1}}\big(\{x\in\rn:\, |u(x)|>t\}\big)\,dt
&\ls \int_0^\infty  \Lambda_{(\infty)}^{n-s}\big(\{x\in\rn:\, |I_s f(x)|>t\}\big)\,dt\\
&\ls \|f\|_{H^1}\\
&\approx [u]_{H^{s,1}},
\end{align*}
as desired.
\end{proof}

To establish the strong inequality for $\text{Cap}_{H_-^{s,1}(n>1)}$, we require the following lemma which generalizes \cite[Proposition~5]{Adams}.

\begin{lemma}\label{lem-3.3}
If $n\ge 2$ and $0\le\beta<\alpha<n$, then
\begin{align*}
 \lf(\int_0^\infty  \Lambda_{(\infty)}^{n-\beta}\big(\{x\in\rn:\, |I_\alpha f(x)|>t^\frac{n-\alpha}{n-\beta}\}\big)\,dt\r)^{\frac{n-\alpha}{n-\beta}} \lesssim\|\vec{R}f\|_{L^1}
 \qquad \ \forall\ \ f\in\cs_{\az,0}.
\end{align*}
\end{lemma}

\begin{proof} Let $f\in\cs_{\az,0}$. Note that Lemma \ref{lem-H1} implies $$f\in\cs_{\az,0}\subset H^1.
		$$
So, by this and the boundedness of each Riesz transform $R_j$ from $H^1$ to $L^1$, we derive $\|\vec{R}f\|_{L^1}<\infty$.
Upon applying \cite[p.\, 118, Corollary]{Adams} we have
\begin{align*}
 &\int_0^\infty  \Lambda_{(\infty)}^{n-\beta}\big(\{x\in\rn:\, |I_\alpha f(x)|>t^\frac{n-\alpha}{n-\beta}\}\big)\,dt \\
&\quad \approx \sup\lf\{
 \int_\rn |I_\alpha f|^\frac{n-\beta}{n-\alpha}\, d\mu:\, \mu\;\textup{nonnegative Radon measre},\, \||\mu\||_{n-\beta}\le 1
 \r\},
\end{align*}
where
$$\||\mu\||_{n-\beta}=\sup_{(x,r)\in\rn\times(0,\infty)} {r^{\beta-n}}\mu\big(B(x,r)\big).$$
Thus, the desired result follows from showing that
$$
\left(\int_\rn |I_\alpha f|^\frac{n-\beta}{n-\alpha}\, d\mu\right)^\frac{n-\alpha}{n-\beta}
\ls \|\vec{R}f\|_{L^1}
$$
holds when $\mu$ is a nonnegative Radon measure on $\rn$ with $\||\mu\||_{n-\beta}\le 1$ - surprisingly - this assertion cannot be extended to the case $\alpha=\beta$ (cf. \cite[Theorem 1.3]{Sp2019} which solves \cite[Open Problem 7.1]{Sp}).

Upon taking $$
\begin{cases}0<\epsilon<\frac{(\alpha-\beta)n}{n-\beta}\\
\bar\beta=n-\beta\\
\bar\alpha=\az-\ez\\
\bar p=\frac{n}{n-\epsilon},
\end{cases}
$$
we have
$$
\bar\beta+\bar\alpha\bar p>n\ \&\ 1<\bar p< n/\bar\alpha.
$$
For such $\bar \az,\bar p$ and $\bar\beta$, we apply  $$\||\mu\||_{\bar\beta}=\||\mu\||_{n-\beta}\le 1$$ and \cite[Theorem 1.1]{LXjde} (extending the main result in \cite{Apisa}) to derive
$$
\lf(\int_\rn \Big(I_{\bar\az} g\Big)^{\frac{\bar \beta\bar p }{n-\bar\alpha\bar p}}\,d\mu\r)^{^{\frac{n-\bar\alpha\bar p}{\bar \beta\bar p }}}\ls \|g\|_{L^{\bar p}}\quad\forall\quad g\in L^{\bar p}.$$
Moreover, it is proved in  \cite[Theorem A]{SSS17} that under the assumption $n\ge2$ one has
$$
\|I_\ez f\|_{L^{\frac n{n-\ez}}}\ls \|\vec Rf\|_{L^1}\qquad\forall\ \ f\in \cs.
$$
From the last two estimates and the fact $$\frac{\bar \beta\bar p }{n-\bar\alpha\bar p}=\frac{n-\beta}{n-\alpha},
$$ it follows that
\begin{align*}
\lf(\int_\rn |I_\alpha f|^\frac{n-\beta}{n-\alpha}\, d\mu\right)^\frac{n-\alpha}{n-\beta}&=
\lf(\int_\rn |I_{\alpha-\epsilon}(I_\epsilon f)|^\frac{n-\beta}{n-\alpha}\, d\mu\right)^\frac{n-\alpha}{n-\beta}
\lesssim \|I_\epsilon f\|_{L^\frac{n}{n-\epsilon}}
\lesssim \|\vec{R}f\|_{L^1},
\end{align*}
as desired; see also \cite[(1.7)]{GRS} for a similar estimate for an elliptic differential operator $\mathbb{A}[D]$.
\end{proof}

Finally, we arrive at the following strong type inequality for $\text{Cap}_{H^{s,1}_-}$.

\begin{theorem}\label{thm3.2}
If $0<\hat{s}<s<1<n$, then
$$
\lf(\int_0^\infty \text{Cap}_{H_-^{\hat{s},1}}\big(\{x\in\rn:\, |u(x)|>t^\frac{n-s}{n-\hat{s}}\}\big)\,dt\r)^\frac{n-s}{n-\hat{s}}\lesssim[u]_{H_-^{s,1}}\qquad \ \forall\ \ u\in \cs.
$$
\end{theorem}

\begin{proof} Given $u\in\cs$. Lemmas \ref{p31} \& \ref{lem-add3} produce
$$
\begin{cases}\text{Cap}_{H^{\hat{s},1}_-}\ls \text{Cap}_{W^{\hat{s},1}} \approx \Lambda_{(\infty)}^{n-\hat{s}}\\
[u]_{H^{s,1}_-} =\|\nabla^s_-u\|_{L^1}=\|\vec{R}(-\Delta)^\frac s2 u\|_{L^1}.
\end{cases}
$$
So, based on the argument for Theorem \ref{thm3.1} and $$f=(-\Delta)^\frac s2 u\in\cs_{s,0}\ \ \text{or}\ \ u=I_s f,$$
it is enough to verify
$$
\lf(\int_0^\infty \text{Cap}_{H_-^{\hat{s},1}}\big(\{x\in\rn:\, |I_s f(x)|>t^\frac{n-s}{n-\hat{s}}\}\big)\,dt\r)^\frac{n-s}{n-\hat{s}}\lesssim\|\vec{R}f\|_{L^1} \qquad \ \forall\ \ f\in \cs_{s,0}.
$$
However, this last estimation is established in Lemma \ref{lem-3.3}.
\end{proof}

\subsubsection*{Weak estimates for $\text{Cap}_{H^{s,1}_\pm}$}\label{s3.3}
Although we do not know whether $\hat{s}$ and $s$ in Theorem \ref{thm3.2} can coincide, we have the following assertion.

\begin{theorem}\label{t31}
Let $0<s<1\le n$. Then
$$
[u]_{H^{s,1}_\pm}\ge\begin{cases}\sup_{t\in (0,\infty)}t\text{Cap}_{H^{s,1}_\pm}\big(\{x\in\rn:\, u(x)>t\}\big)\\
\sup_{t\in (0,\infty)}t\text{Cap}_{H^{s,1}_\pm}\big(\{x\in\rn:\, u(x)<-t\}\big),
\end{cases}
\qquad \ \forall\ \ u\in \mathcal{S}.
$$
But, if $X\in\big\{H^{s,1}_+,H^{s,1}_-(n=1)\big\}$ then there is no constant $C>0$ such that
$$
\int_0^\infty \text{Cap}_{X}\big(\{x\in\rn:\, |u(x)|>t\}\big)\,dt\le C[u]_{X}\qquad \ \forall\ \ u\in \mathcal{S}.
$$
\end{theorem}

\begin{proof}
For $(t,u)\in(0,\infty)\times\mathcal{S}$, since $$\{x\in\rn:\, u(x)>t\}$$ is open, by the definition of $\text{Cap}_{X}$, for any $\epsilon\in(0,\infty)$ there exists a compact set $$K\subset \{x\in\rn:\, u(x)>t\}$$ such that
$$
\text{Cap}_{H^{s,1}_\pm}\big(\{x\in\rn:\, u(x)>t\}\big)< \text{Cap}_{H^{s,1}_\pm}(K)+\ez.
$$
Let
$v=
t^{-1}u$. Then
$$
v\in \mathcal{S}\ \ \&\ \ v>1\ \text{on}\ K.
$$
Accordingly, by definition we have
$$
\text{Cap}_{H^{s,1}_\pm}(K)\le [v]_X= t^{-1}[u]_{H^{s,1}_\pm},
$$
which implies
$$\text{Cap}_{H^{s,1}_\pm}\big(\{x\in\rn:\, u(x)>t\}\big)< t^{-1}[u]_{H^{s,1}_\pm}+\ez.$$
Letting $\ez\to0$ gives the desired estimate
$$\sup_{t\in (0,\infty)}t\text{Cap}_{H^{s,1}_\pm}\big(\{x\in\rn:\, u(x)>t\}\big)\le [u]_X.$$
Since $$u(x)<-t\Leftrightarrow-u(x)>t\ \ \&\ \ [-u]_{H^{s,1}_\pm}=[u]_{H^{s,1}_\pm},
$$
we get
$$
\text{Cap}_{H^{s,1}_\pm}\big(\{x\in\rn:\, u(x)<-t\}\big)\le t^{-1}[u]_{H^{s,1}_\pm}.
$$

In order to verify the nonexistence of the capacitary strong estimate for $X$ under consideration,
we note  that
$$
\|u\|_{L^{\frac{n}{n-s},\infty}}=\sup_{t>0}t\big|\big\{x\in\rn: |u(x)|>t\big\}\big|^\frac{n-s}{n}\lesssim [u]_{H^{s,1}_\pm}\qquad \ \forall\ \ u\in \mathcal{S},
$$
which 
follows from
\begin{equation}
\label{eww}
\|I_sf\|_{L^{\frac{n}{n-s},\infty}}\lesssim \min\Big\{\|f\|_{L^1},\|\vec{R}f\|_{L^1}\Big\}\qquad \ \forall\ \ f\in {(-\Delta)^\frac s2 \cs \subset \mathcal{S}_{s,0}}.
\end{equation}
In fact,
$$\|I_sf\|_{L^{\frac{n}{n-s},\infty}}\lesssim \|f\|_{L^1}
$$
can be seen from \cite{Aduke}.
This last inequality, along with the fact (from the Fourier transform) that
$$
I_sf=-\sum_{j=1}^n R_j^2I_sf=-\sum_{j=1}^n R_jI_s(R_jf){\quad  \text{almost everywhere on }\; \rn}
$$
and (cf. \cite[(1.5)]{O})
$$
\|R_ju\|_{L^{\frac{n}{n-s},\infty}}\lesssim \|u\|_{L^{\frac{n}{n-s},\infty}},
$$
in turn derives
$$
\|I_sf\|_{L^{\frac{n}{n-s},\infty}}\lesssim \|\vec{R}f\|_{L^1}.
$$
Both \eqref{eww} and the definition of $\text{Cap}_{H^{s,1}_\pm}$
give the iso-capacitary inequality
\begin{equation}
\label{isoc}
|K|^\frac{n-s}{n}\lesssim \text{Cap}_{H^{s,1}_\pm}(K)\quad\forall\quad \text{compact}\quad K\subset\rn.
\end{equation}

We use \eqref{isoc} to verify the failure of the capacitary strong estimate for $X\in\big\{H^{s,1}_+,H^{s,1}_-(n=1)\big\}$. Suppose otherwise that
$$\exists\ C_0>0\ \ \text{such that}\ \
\int_0^\infty \text{Cap}_{X}\big(\{x\in\rn:\, |u(x)|>t\}\big)\,dt\le C_0[u]_{X}\qquad \ \forall\ \ u\in \mathcal{S}.
$$
Then, a standard layer-cake method (cf. \cite[p.101]{LXYY}) and \eqref{isoc} derive a constant $C_1$ depending on $C_0$ such that any $u\in\cs$ satisfies
\begin{align*}
\|u\|^\frac{n}{n-s}_{L^\frac{n}{n-s}}&=\int_0^\infty|\{x\in\rn: |u(x)|>t\}|\,dt^\frac{n}{n-s}\\
&\ls  \left(\int_0^\infty|\{x\in\rn: |u(x)|>t\}|^\frac{n-s}{n}\,dt\right)^\frac{n}{n-s}\\
&\lesssim \left(\int_0^\infty\text{Cap}_{X}\big(\{x\in\rn: |u(x)|>t\}\big)\,dt\right)^\frac{n}{n-s}\\
&\lesssim \big(C_1[u]_{X}\big)^\frac{n}{n-s}.
\end{align*}
This contradicts the failure of \eqref{e4} mentioned in \S \ref{s12}, thereby completing the verification.

\end{proof}

\subsection{Restrictions/traces of $H^{s,1}\ \&\ H^{s,1}_\pm$}\label{s-trace}

Being motivated by \cite[Theorem 1.4]{Xadv} for $W^{s,1}$, we establish the coming-up-next restricting/tracing principle.

\begin{theorem}\label{thm1}
Let $0<s<1\le n$, $\mu$ be a nonnegative Radon measure on $\mathbb R^n$ and
$$
\begin{cases}\|u\|_{L^\frac{n}{n-s}(\mu)}=\bigg(\int_{\rn}|u|^\frac{n}{n-s}\,d\mu\bigg)^\frac{n-s}{n}\\
\|u\|_{L^{\frac{n}{n-s},\infty}(\mu)}=\sup_{t>0}t\mu\Big(\big\{x\in\rn: |u(x)|>t\big\}\Big)^\frac{n-s}{n}.
\end{cases}
$$
Then the following two assertions are equivalent:
\begin{enumerate}
\item[\rm (i)] there exists a positive constant $c$ such that
\begin{equation*}
\big(\mu(K)\big)^\frac{n-s}{n}\le c\text{Cap}_{X}(K)\qquad \ \forall\ \ \text{compact}\; K\subset\mathbb R^n.
\end{equation*}
\item[\rm (ii)] there exists a positive constant $C$ such that
\begin{equation*}
C[u]_X\ge\begin{cases}\|u\|_{L^\frac{n}{n-s}(\mu)}\ \ &\text{as}\ \ X=H^{s,1}\\

\|u\|_{L^{\frac{n}{n-s},\infty}(\mu)}\ \ &\text{as}\ \ X=H^{s,1}_\pm
\end{cases}
 \qquad \ \forall\ u\in \mathcal{S}.
\end{equation*}
\end{enumerate}
Moreover, the constants $c$ and $C$ are comparable to each other.
\end{theorem}

\begin{proof} The consequence part of Theorem \ref{thm1} follows from (i)$\Leftrightarrow$(ii) and
	$$
	[u]_X\gtrsim \begin{cases}\|u\|_{L^\frac{n}{n-s}}\ \ &\text{as}\ \ X=H^{s,1}\\
		
		\|u\|_{L^{\frac{n}{n-s},\infty}}\ \ &\text{as}\ \ X=H^{s,1}_\pm
	\end{cases}
	\qquad \ \forall\ u\in \mathcal{S}.
$$
	So, we are required to validate 	(i)$\Leftrightarrow$(ii).		
Two cases are considered for
	$$
	\begin{cases}
u\in \mathcal{S}\\
	t\in(0,\fz)\\
	E_t=\big\{x\in\rn:\, |u(x)|>t\big\}\\
	E_{t,+}=\big\{x\in\rn:\, u(x)>t\big\}\\
E_{t,-}=\big\{x\in\rn:\, u(x)<-t\big\}.
	\end{cases}
	$$
	
	{\it Case 1}: (i)$\Leftrightarrow$(ii) for $X=H^{s,1}_\pm$.
	
	On the one hand, if (i) holds, then the subadditivity of $\mu$, the decomposition
	$$E_t=E_{t,+}\cup E_{t,-},$$ and Theorem \ref{t31} derive
	\begin{align*}
	\big(\mu(E_t)\big)^\frac{n-s}{n}&\le\big(\mu(E_{t,+})+\mu(E_{t,-})\big)^\frac{n-s}{n}\\
	&\le {\big(\mu(E_{t,+})\big)}^\frac{n-s}{n}+\big(\mu(E_{t,-})\big)^\frac{n-s}{n}\\
	&\lesssim \text{Cap}_X(E_{t,+})+\text{Cap}_X(E_{t,-})\\
	&\lesssim t^{-1}[u]_X,
	\end{align*}
	thereby verifying (ii).
	
	On the other hand, suppose that (ii) is valid. For any compact $K\subset\rn$ let
	$$
	u\in\mathcal{S}\ \ \&\ \ u\ge 1\ \ \text{on}\ \ K.
	$$
	Then
$$	
\big(\mu(K)\big)^\frac{n-s}{n}\le\big(\mu(E_{1,+})\big)^\frac{n-s}{n}\le\big(\mu(E_1)\big)^\frac{n-s}{n}\lesssim [u]_X.
$$
Accordingly, by definition we reach (i).

{\it Case 2:} (i)$\Leftrightarrow$(ii) for $X=H^{s,1}$.

On the one hand, for any $k\in\zz$ the open set $E_{2^k}$ has a compact subset $K_k$ such that
$$\mu(E_{2^k})\le 2\mu(K_k).$$
Thus, if (i) is valid, then
\begin{align*}
\int_{\mathbb R^n}|u|^\frac{n}{n-s}\,d\mu
& =\sum_{k\in\zz}\int_{2^k}^{2^{k+1}} \mu(E_t)\, dt^\frac{n}{n-s}\nonumber\\
&\le (2^\frac{n}{n-s}-1) \sum_{k\in\zz} 2^{\frac{kn}{n-s}} \mu(E_{2^k})\notag\nonumber\\
&\le 2^{\frac{n}{n-s}+1} \sum_{k\in\zz} 2^{\frac{kn}{n-s}} \mu(K_k)\\
&\le c^\frac{n}{n-s}2^{\frac{n}{n-s}+1} \sum_{k\in\zz} 2^{\frac{kn}{n-s}}\big(\text{Cap}_{X}(K_k)\big)^{\frac{n}{n-s}}\notag\\
&\le c^\frac{n}{n-s} 2^{\frac{n}{n-s}+1} \sum_{k\in\zz} 2^{\frac{kn}{n-s}}\big(\text{Cap}_{X}(E_{2^k})\big)^{\frac{n}{n-s}} \notag
\end{align*}
Note that for any {nonnegative} sequence $\{a_j\}_{j\in\zz}$,
$$
\lf(\sum_{j\in\zz} a_j\r)^\kz\le \sum_{j\in\zz} a_j^\kz
\quad\forall\ \ \kz\in(0,1].
$$
This in turn gives
\begin{align*}
\sum_{k\in\zz} 2^{\frac{kn}{n-s}}\big(\text{Cap}_{X}(E_{2^k})\big)^{\frac{n}{n-s}}
\le  \lf(\sum_{k\in\zz} 2^{k}\,\text{Cap}_{X}(E_{2^k})\r)^\frac{n}{n-s}.
\end{align*}
Moreover, by Lemma \ref{p31}(ii) it follows that
\begin{align*}
\sum_{k\in\zz} 2^{k}\,\text{Cap}_{X}(E_{2^k})
& = 2\sum_{k\in\zz} \int_{2^{k-1}}^{2^k} {\text{Cap}_{X}}(E_{2^k})\,dt\notag\\
&\le 2\sum_{k\in\zz} \int_{2^{k-1}}^{2^k} \text{Cap}_{X}(E_t)\,dt\\
&= 2
\int_0^\infty \text{Cap}_{X}(E_t)\,dt. \notag
\end{align*}
Altogether, we use Theorem \ref{thm3.1}  to obtain
\begin{align*}
\int_{\mathbb R^n}|u|^\frac{n}{n-s}\,d\mu
\ls  \lf(
\int_0^\infty \text{Cap}_{X}(E_t)\,dt \r)^\frac{n}{n-s} \lesssim [u]_{X}^\frac{n}{n-s},
\end{align*}
which implies (ii).

On the other hand, suppose that (ii) is true. Upon letting $K$ be a  compact subset of $\rn$ we gain that for any $u\in \mathcal{S}$ with $u\ge 1$ on $K$,
\begin{align*}
\big(\mu(K)\big)^\frac{n-s}{n}\le \left(\int_{\mathbb R^n}|u|^\frac{n}{n-s}\,d\mu\right)^\frac{n-s}{n}\lesssim [u]_{X}.
\end{align*}
Via taking the supremum over all such $u\in \mathcal{S}$ with $u\ge 1$ on $K$ we get (i).
\end{proof}

\section{Duality laws for $H^{s,1}\, \&\, \mathring{H}^{s,1}_\pm$}\label{s5}

\subsection{Adjoint operators of $\nabla^s_\pm$ via $\{\mathcal{S},\BMO\}$}\label{s51}
This subsection describes the adjoint operators of $\nabla^s_{\pm}$ (existing as two basic notions in fractional vector calculus).

\subsubsection*{Integration-by-parts}\label{s511} Below is a two-fold computation.

\begin{itemize}
	
	\item[$\rhd$] On the one hand, the dual operator $\big[(-\Delta)^\frac{s}{2}\big]^\ast$ of $(-\Delta)^\frac{s}{2}$ is itself, i.e., $$[\nabla^s_+]^\ast=\nabla^s_+,$$
in the sense of \begin{align*}
\laz [\nabla^s_+]^\ast f,\,\phi\raz=\laz f,\,\nabla^s_+ \phi \raz =\laz \nabla^s_+ f,\,\phi\raz
\qquad \ \forall\, (f,\phi)\in\cs_s'\times\cs.
\end{align*}
This is reasonable because of (cf. \cite{S})
	\begin{equation*}
	\label{eIBP1}
	\begin{cases}
	\int_{\mathbb R^n}(-\Delta)^\frac{s}{2}f(x)\phi(x)\,dx=\int_{\mathbb R^n}f(x)(-\Delta)^\frac{s}{2}\phi(x)\,dx\\
	\int_{\mathbb R^n}f(x)I_s\phi(x)\,dx=\int_{\mathbb R^n}I_sf(x)\phi(x)\,dx
	\end{cases}
	\ \forall\ (f,\phi)\in (C_c^\infty)^2
	\end{equation*}
	and
	$$
	(-\Delta)^\frac{s}{2}\big((-\Delta)^\frac{s}{2}u\big)=(-\Delta)^{s}u\qquad \forall\ \ u\in C_c^\infty.
	$$

	\item[$\rhd$] On the other hand, if we define
\begin{align*}\label{eq-div}
	\text{div}^s\vec{g}=(-\Delta)^{\frac s2}\vec{R}\cdot\vec{g}
\end{align*}
then  it enjoys (cf. \cite[Theorem 1.3]{SS1})
	$$
	-\text{div}^s(\nabla^s_- u)=(-\Delta)^su \qquad \forall\ \  u\in C_c^\infty
	$$
	and (cf. \cite[Lemma 2.5]{CS})
\begin{equation*}
	\label{eIBP2}\int_{\mathbb R^n}f(x)(-\text{div}^s\vec{g})(x)\,dx=\int_{\mathbb R^n}\vec{g}(x)\cdot \nabla^s_- f(x)\,dx\qquad \forall\ \ (f,\vec{g})\in C_c^\infty\times (C_c^\infty)^n.
	\end{equation*}
Thus $-\text{div}^s$ exists as the dual operator $\big[\nabla^s_{-}\big]^\ast$  of $\nabla^s_{-}$, i.e., $$[\nabla^s_-]^\ast=-\text{div}^s.$$
\end{itemize}

\subsubsection*{Dual pairing for $\{\mathcal S,\BMO\}$}\label{s512} We are required to verify that $\BMO$ can be embedded in a family of relatively bigger spaces.

\begin{lemma}\label{lem-add2}
If $s\in(0,1)$, then $\BMO\subset\cs_s'$.
\end{lemma}

\begin{proof} 
In order to verify $f\in\cs_s'$, it suffices to show that $f\in\BMO$ induces a continuous linear functional on $\cs_s$. To this end, we consider
$$
L_f(\phi)=\int_\rn f(x)\phi(x)\,dx\qquad\forall\ \ \phi\in\cs_s.
$$
Upon writing
\begin{align*}
\lf|\int_\rn f(x)\phi(x)\,dx\r|
&\le \int_{B(0,1)}|f(x)\phi(x)|\, dx
+\sum_{j=1}^\infty \int_{2^{j-1}\le |x|<2^j} |f(x)\phi(x)|\, dx,
\end{align*}
and noting both
\begin{align*}
\int_{B(0,1)}|f(x)\phi(x)|\, dx
&\le \|\phi\|_{L^\infty} \int_{B(0,1)} |f(x)|\,dx\\
&\ls \|\phi\|_{L^\infty}  \lf(\frac1{|B(0,1)|}\int_{B(0,1)} |f(x)-f_{B(0,1)}|\,dx +|f_{B(0,1)}|\r)\\
&\ls  \rho_{n+s,0}(\phi)\lf(\|f\|_\BMO +|f_{B(0,1)}|\r)
\end{align*}
and
\begin{align*}
\int_{2^{j-1}\le |x|<2^j}|f(x)\phi(x)|\, dx
&\ls \rho_{n+s,0}(\phi) \int_{2^{j-1}\le |x|<2^j}\frac{|f(x)}{1+|x|^{n+s}}\, dx\\
&\ls  \rho_{n+s,0}(\phi) 2^{-js}\lf(\frac1{|B(0,2^j)|}\int_{B(0,2^j)} |f(x)|\,dx \r)\\
&\ls  \rho_{n+s,0}(\phi)  2^{-js}\lf({\sum_{i=0}^j}\frac1{{|B(0,2^i)|}}\int_{B(0,2^i)} |f(x)-f_{B(0,2^i)}|\,dx +|f_{B(0,1)}|\r)\\
&\ls \rho_{n+s,0}(\phi) j 2^{-js} \lf(\|f\|_\BMO +|f_{B(0,1)}|\r),
\end{align*}
we obtain
\begin{align*}
\lf|\int_\rn f(x)\phi(x)\,dx\r|
&\ls \rho_{n+s,0}(\phi) \lf(1+ \sum_{j=1}^\infty j 2^{-js} \r) \lf(\|f\|_\BMO +|f_{B(0,1)}|\r)\\
&\ls \rho_{n+s,0}(\phi)\lf(\|f\|_\BMO +|f_{B(0,1)}|\r),
\end{align*}
as desired.
\end{proof}

\begin{proposition}\label{prop-add}
For $s\in(0,1)$ one has the following two implications.
\begin{enumerate}
\item[\rm(i)] If $(f,\phi)\in \BMO\times\cs$, then
$$\laz \nabla^s_+ f,\,\phi\raz=\laz f,\, \nabla^s_+ \phi \raz.$$

\item[\rm(ii)] If $\vec U=(U_1,\dots, U_n)\in (L^\infty)^n$ and $\phi\in\cs$, then
$$\laz [\nabla^s_-]^\ast \vec U,\,\phi\raz= \sum_{j=1}^n\laz U_j,\, \nabla^s_j \phi \raz.$$

\end{enumerate}

\end{proposition}

\begin{proof}
Note that (i) follows directly from Lemma \ref{lem-add2} and {Definition \ref{defn1}(i)}.

Now we show (ii). For any $j\in\{1,2,\dots,n\}$, it is known that $R_j$ maps $L^\infty$ functions continuously into $\BMO$ and that
$$R_jU_j\in\BMO\subset\cs_s'$$
follows from Lemma \ref{lem-add2}. So, {Definition \ref{defn1}(i)} derives that
every $$(-\Delta)^{\frac s2}R_jU_j{\in\cs'}.$$
By this and the definition of $[\nabla^s_-]^\ast$, we have
$$[\nabla^s_-]^\ast \vec U=-\text{div}^s\vec U=-\sum_{j=1}^n(-\Delta)^{\frac s2}R_jU_j\in\cs'.$$
Thus, for $\phi\in\cs$ we have
\begin{align*}
\laz [\nabla^s_-]^\ast \vec U,\,\phi\raz
=-\sum_{j=1}^n\laz (-\Delta)^{\frac s2}R_jU_j,\, \phi\raz
=-\sum_{j=1}^n\laz R_jU_j,\,  (-\Delta)^{\frac s2}\phi\raz.
\end{align*}
Since $\phi\in\cs$, Lemma \ref{lem-H1} yields $$(-\Delta)^{\frac s2}\phi\in H^1.
$$
By
$$[H^1]^\ast=\BMO\ \ \&\ \ R_j^\ast=-R_j,
$$
we further obtain
$$ \laz R_jU_j,\,  (-\Delta)^{\frac s2}\phi\raz
=  \laz U_j,\,   R_j^\ast(-\Delta)^{\frac s2}\phi\raz
=- \laz U_j,\,   R_j(-\Delta)^{\frac s2}\phi\raz
=-\laz U_j,\,   \nabla^s_j\phi\raz,$$
thereby finding
$$
\laz [\nabla^s_-]^\ast \vec U,\,\phi\raz=\sum_{j=1}^n\laz U_j,\,   \nabla^s_j\phi\raz.
$$
\end{proof}

\subsection{Dualities of $H^{s,1}\ \&\ \mathring H^{s,1}_\pm$}\label{s52} This subsection is divided into two parts.

\subsubsection*{Fundamental duality}\label{s521}
Below is the expected duality law.

\begin{theorem}\label{thm2}
	Let $0<s<1\le n$, $T\in\cs'$ and $\nu$ be a nonnegative Radon measure on $\mathbb R^n$. Then:
 \begin{enumerate}
 \item[\rm (i)] $T\in [{H}^{s,1}]^\ast$ if and only if
 $$\exists\ (U_0,U_1,...,U_n)\in\big( L^\infty\big)^{1+n}\ \text{such that}\
	T=[\nabla^s_+]^\ast U_0+[\nabla^s_-]^\ast(U_1,...,U_n)\ \ \text{in}\ \ \cs'
	$$
	if and only if
	$$T\in (-\Delta)^\frac{s}{2}\BMO.$$

\item[\rm (ii)] $\||\nu\||_{n-s}<\infty$ if and only if
$$
\int_{\rn}|f|\,d\nu\lesssim [f]_{H^{s,1}}\ \ \forall\ \ f\in\cs_\infty
$$
if and only if
$$
\int_{\rn}|f|\,d\nu\lesssim [f]_{W^{s,1}}\ \ \forall\ \ f\in\cs_\infty.
$$
\item[\rm (iii)]
{$T\in [\mathring {H}^{s,1}_+]^\ast$ if and only if
$$\exists \ U_0\in L^\infty
	\ \ \text{
	such that}\ \
	T= [\nabla^s_+]^\ast U_0 \quad \text{in}\quad \cs'.
	$$}

\item[\rm (iv)] {$T\in [\mathring {H}^{s,1}_-]^\ast$ if and only if
$$\exists\
	 \vec{U}=(U_1,...,U_n)\in\big( L^\infty\big)^n
		\ \ \text{
	such that}\ \
	T=
	[\nabla^s_-]^\ast\vec{U} \quad \text{in}\quad \cs'.
	$$}
 \end{enumerate}
\end{theorem}

\begin{proof}(i) First of all, by using the density of $\cs_\infty$ in {both $H^1$ and} $H^{s,1}$ (cf. Proposition \ref{prop-dense1}) and
the invariant of $\cs_\infty$ under $I_s$ and $(-\Delta)^\frac s2$, we have
\begin{align*}
T\in [H^{s,1}]^\ast
&\Leftrightarrow |T(f)| \ls [f]_{H^{s,1}}\;\forall\; f\in \cs_\infty\\
&\Leftrightarrow |T(I_s g)| \ls \|g\|_{H^1}\;\forall\; g=(-\Delta)^\frac s2 f\in \cs_\infty\\
&\Leftrightarrow T\circ I_s \in [H^1]^\ast.
\end{align*}
Consequently, an application of the Fefferman-Stein duality and decomposition (cf. \cite[Theorem 2 \& Theorem 3]{FS})
	$$
	[H^1]^\ast=\BMO=L^\infty+\vec{R}\cdot (L^\infty)^n,
	$$
produces some
$$
(U_0,U_1,...,U_n)\in\big( L^\infty\big)^{1+n}
$$
such that
\begin{align}\label{eq-x3}
T\in [H^{s,1}]^\ast
\Leftrightarrow T\circ I_s &= U_0+\sum_{j=1}^n R_jU_j.
\end{align}

Next, we utilize \eqref{eq-x3} to show the equivalence in (i). Let $T\in [H^{s,1}]^\ast$.
For any $\phi\in\cs$, {if we let
$$\psi=(-\Delta)^\frac s2 \phi,$$
then  Lemmas \ref{lem-x1} \& \ref{lem-H1} imply}
$$
\begin{cases}
\phi=I_s\psi\\
\psi\in\cs_s\cap H^1\\
\laz T, \phi\raz
=T(\phi)=T(I_s\psi)=\lf(T\circ I_s\r)(\psi)=\laz T\circ I_s,\psi\raz.
\end{cases}
$$
Upon applying \eqref{eq-x3},
$$R_jU_j\in\BMO\subset\cs_s',
$$
Proposition \ref{prop-add}(i) and {Definition \ref{defn1}(i)},  we arrive at
\begin{align*}
\laz T\circ I_s,\psi\raz
&= \lf\laz U_0+\sum_{j=1}^n R_jU_j, (-\Delta)^\frac s2 \phi\r\raz\\
&=\lf\laz (-\Delta)^\frac s2 U_0+\sum_{j=1}^n (-\Delta)^\frac s2 R_jU_j,  \phi\r\raz.
\end{align*}
This in turn gives
\begin{align}\label{eq-x4}
T=(-\Delta)^\frac s2 U_0+\sum_{j=1}^n (-\Delta)^\frac s2 R_jU_j=[\nabla^s_+]^\ast U_0+[\nabla^s_-]^\ast(U_1,...,U_n)\ \ \textup{in}\ \ \cs'
\end{align}
and so
$$T\in (-\Delta)^\frac s2\BMO.$$

Conversely,  we assume that
$$T\in (-\Delta)^\frac s2\BMO\ \ \text{or}\ \
\eqref{eq-x4}\ \ \text{holds for some}\ \
(U_0,U_1,...,U_n)\in\big( L^\infty\big)^{1+n}.
$$
Then, for any $\psi\in\cs_\infty$, we have
$$\phi=I_s\psi\in\cs_\infty\ \ \text{and}\ \  \psi=(-\Delta)^\frac s2 \phi,
$$
which, combined with the facts
$$U_0\in L^\infty\subset\cs_s'\ \ \&\ \ R_jU_j\in\BMO\subset\cs_s'
$$
and Definition \ref{defn1}(i), yields
\begin{align*}
\laz T\circ I_s, \psi\raz
&= (T\circ I_s)(\psi)\\
&=T(I_s\psi)\\
&=T(\phi)\\
&=\lf\laz (-\Delta)^\frac s2 U_0+\sum_{j=1}^n (-\Delta)^\frac s2 R_jU_j,\phi\r\raz\\
&=\lf\laz  U_0+\sum_{j=1}^n R_jU_j,(-\Delta)^\frac s2\phi\r\raz\\
&=\lf\laz  U_0+\sum_{j=1}^n R_jU_j, \psi\r\raz.
\end{align*}
Due to the density of $\cs_\infty$ in $H^1$ {and $[H^1]^\ast=\BMO$}, the last series of identities implies
$$
T\circ I_s= U_0+\sum_{j=1}^n R_jU_j\ \ \text{in}\ \  \BMO.
$$
Combining this and \eqref{eq-x3}  yields $$T\in [H^{s,1}]^\ast.$$

(ii) Noting that $\cs_\infty$ is dense in both $H^{s,1}$ and $W^{s,1}$ as shown in Proposition \ref{prop-dense1}, we apply \eqref{e2} to deduce  $$W^{s,1}\subset H^{s,1}\ \ \text{with}\ \ [f]_{H^{s,1}}\lesssim [f]_{W^{s,1}}.
$$ Accordingly, the desired equivalence follows from \cite[Proposition~3.2]{LX} and
\begin{equation*}\label{e41}
\int_\rn |f|\, d\nu
=\int_\rn \lf|I_s(-\Delta)^\frac s2 f\r|\,d\nu
\ls \||\nu\||_{n-s}\|(-\Delta)^\frac s2 f\|_{H^1}\approx \||\nu\||_{n-s}[f]_{H^{s,1}}\qquad \ \forall\ \ f\in\cs_\infty.
\end{equation*}

{(iii) Let $T\in\cs'$.} If
$$T=[\nabla^s_+]^\ast U_0\ \ \text{in}\ \  \cs'\ \ \text{ for some}\ \  U_0\in L^\infty,
$$
then
\begin{align*}
\laz T, \phi\raz
=\laz (-\Delta)^{\frac s2} U_0, \phi\raz
=\laz  U_0, (-\Delta)^{\frac s2}\phi\raz\qquad \ \forall\ \ \phi\in\cs,
\end{align*}
where the second equality holds thanks to $L^\infty\subset\cs_s'$ and {Definition \ref{defn1}(i)}. Thus,
\begin{align*}
|\laz T, \phi\raz| \le \|U_0\|_{L^\infty} \| (-\Delta)^{\frac s2}\phi\|_{L^1}=\|U_0\|_{L^\infty}\|\phi\|_{H^{s,1}_+} \qquad \ \forall\ \ \phi\in\cs,
\end{align*}
which implies that  $T$ induces a bounded linear functional on $\mathring H^{s,1}_+$ in terms of the density of   $\cs$ in $\mathring H^{s,1}_+$.

To obtain the converse part, assuming
$$T\in [\mathring H^{s,1}_+]^\ast,
$$
we are about to find $$U_0\in L^\infty\ \ \text{such that}\ \
T=[\nabla^s_+]^\ast U_0\ \ \text{in}\ \  \cs'.
$$
Motivated by the argument in \cite[p.\,399]{BB}, we consider the bounded linear operator
\begin{align*}
A_+:\ \   &\mathring H^{s,1}_+\to L^1\\
&u\mapsto (-\Delta)^{\frac s2}u
\end{align*}
which is actually a closed operator thanks to the definition of $\mathring H^{s,1}_+$ based on the completeness of $\mathcal{S}$ in $H^{s,1}_+$.
If $$u\in \mathring H^{s,1}_+\ \ \text{obeys}\ \ \|(-\Delta)^{\frac s2}u\|_{L^1}=0,$$
then
$$
(-\Delta)^{\frac s2}u=0\ \ \text{almost everywhere on}\ \ \rn,
$$
which implies
$$\laz u, \phi\raz=\laz u, (-\Delta)^\frac s2 I_s \phi\raz=\laz (-\Delta)^{\frac s2}u, I_s\phi\raz=0\qquad\forall\ \ \phi\in\cs_\infty,$$
that is,
$u=0$ in $\cs'/\mathcal P$, or equivalently, $u$ is a polynomial on $\rn$.
Further, any $u\in\cs_s'$ being a polynomial forces $u$ to be
a constant function on $\rn$. In other words, it holds $u=0$ in $\mathring H^{s,1}_+$.
Thus, the operator $A_+$ is injective. In the meantime, $A_+$ enjoys
$$
\|A_+u\|_{L^1}=\|(-\Delta)^\frac{s}{2}u\|_{L^1}=\|u\|_{H^{s,1}_+}\quad\forall\quad u\in \mathring H^{s,1}_+.
$$
Consequently, $A_+$ has a continuous inverse from $L^1$ to $\mathring H^{s,1}_+$. Since the definition of  $\mathring H^{s,1}_+$ (determined by the closure of $\mathcal{S}$ in $H^{s,1}_+$) ensures that $$A_+: \mathring H^{s,1}_+\to L^1$$ is a closed linear operator, the closed range theorem (see \cite[p.\,208, Corollary~1]{Y}) derives that the transpose of $A_+$
$$ A_+^\ast:\, L^\infty \to [\mathring H^{s,1}_+]^\ast,$$
defined by
$$\laz   A_+^\ast  F, u\raz=\laz  F,   A_+ u\raz\quad\forall\ \vec F\in L^\infty\,\&\, u\in \mathring H^{s,1}_+,$$
is surjective. In particular, since $$T\in [\mathring H^{s,1}_+]^\ast,
$$
we can find $$U_0\in L^\infty\ \ \text{such that}\ \
  A_+^\ast  U_0 = T.$$
Consequently, for any $u\in\cs$, we have
\begin{align*}
\laz   A_+^\ast U_0,\, u\raz=\laz U_0,   A_+ u\raz = \laz U_0,  A_+ u\raz =  \laz U_0,  (-\Delta)^\frac s2 u\raz
=\laz [\nabla^s_+]^\ast U_0,   u\raz,
\end{align*}
whence gives
$$
T=  A_+^\ast U_0 = [\nabla^s_+]^\ast U_0 \ \ \text{in}\ \ \cs'.
$$

{(iv)} Let $T\in\cs'$. If
$$
T=[\nabla^s_-]^\ast \vec U\ \ \text{in}\ \  \cs'\ \ \text{for some}\ \  \vec U=(U_1,\dots, U_n)\in (L^\infty)^n,
$$
then Proposition \ref{prop-add}(ii) implies
\begin{align*}
\laz T, \phi\raz
=\laz [\nabla^s_-]^\ast \vec U, \phi\raz
=\sum_{j=1}^n\laz (-\Delta)^\frac s2R_jU_j, \phi\raz
=\sum_{j=1}^n\laz   U_j, \nabla^s_j\phi\raz\quad \ \forall\quad\phi\in\cs,
\end{align*}
and hence
\begin{align*}
|\laz T, \phi\raz| \le \sum_{j=1}^n \|U_j\|_{L^\infty} \|\nabla^s_j\phi\|_{L^1} \quad\forall\quad\phi\in\cs.
\end{align*}
Since $\cs$ is dense in $\mathring H^{s,1}_-$, $T$ induces a bounded linear functional on $\mathring H^{s,1}_-$.

To obtain the converse part, assuming $$T\in [\mathring H^{s,1}_-]^\ast
$$
we are about to show
$$
T=[\nabla^s_-]^\ast \vec U\ \ \text{for some}\ \  \vec U\in (L^\infty)^n.
$$

To this end, we consider the bounded linear operator
\begin{align*}
A_-:\ \   &\mathring H^{s,1}_-\to (L^1)^n\\
&u\mapsto \nabla_-^su.
\end{align*}

\begin{itemize}
\item [$\rhd$] Suppose
$$u\in \mathring H^{s,1}_-\ \ \text{obeys}\ \  \nabla^s_- u=0\ \ \text{in}\ \
(L^1)^n.
$$
Since $u\in \mathring H^{s,1}_-$, it follows that  $u\in\cs_s'$.
For any $\psi\in\cs_\infty$,  the Fourier transform implies that
$$ \psi=-\sum_{j=1}^n \nabla^s_j I_s R_j\psi\ \ \text{holds with}\ \ I_s R_j\psi\in\cs_\infty\subset L^\infty,$$
thereby giving
\begin{align*}
\lf|\laz u, \psi\raz\r|
=\lf|\sum_{j=1}^n \laz u,  \nabla^s_j I_s R_j\psi\raz\r|
=\lf|\sum_{j=1}^n \laz \nabla^s_j  u,  I_s R_j\psi\raz\r|
\le \sum_{j=1}^n \|\nabla^s_j  u\|_{L^1} \|I_s R_j\psi\|_{L^\infty}
=0.
\end{align*}
This shows that
$u=0$ in $\cs'/\mathcal P$. In other words, $u$ is a polynomial on $\rn$.
However, if a polynomial $u$ is a bounded linear functional on $\cs_s$,
then $u$ must be a constant function, which  implies that $u=0$ in $\mathring H^{s,1}_-$. In other words,
$$
A_-: \mathring H^{s,1}_-\to (L^1)^n
$$
is injective.

\item[$\rhd$] This last injectiveness and the next identification
$$
\|A_- u\|_{(L^1)^n}=\|\nabla_-^s u\|_{L^1}=\|u\|_{H^{s,1}_-}\quad\forall\quad u\in\mathring H^{s,1}_-.
$$
derive that
$$A_-: \mathring H^{s,1}_-\to \mathcal{R}(A_-)=A_-(\mathring H^{s,1}_-)
$$
has a continuous inverse sending $\mathcal{R}(A_-)$ to $\mathring H^{s,1}_-$.

\item[$\rhd$] Clearly, the closure of $\mathcal{S}$ in $H^{s,1}_-$ ensures that $\mathcal{R}(A_-)$ is closed in $(L^1)^n$. So, from the closed range theorem it follows that the $A_-$'s transpose
\begin{align*}
 A_-^\ast:\ \ &[\mathcal{R}(A_-)]^\ast \to [\mathring H ^{s,1}_- ]^\ast\\
 &\vec{F}\mapsto A_-^\ast\vec{F}.
 \end{align*}
defined by
$$
\laz   A_-^\ast \vec F, \phi\raz=\laz \vec F,   A_- \phi\raz=\laz \vec F,   \nabla^s_- \phi\raz\quad\forall\quad \phi\in \mathcal{S},
$$
is surjective. Consequently, for the hypothesis
$
T\in[\mathring H^{s,1}_- ]^\ast
$
there exists
$$\vec U_o\in [\mathcal{R}(A_-)]^\ast\ \ \text{such that}\ \
  A_-^\ast \vec U_o = T\ \ \text{in}\ \ \mathcal{S}'.
  $$
  Although it is uncertain that $\vec{U}_o\in (L^\infty)^n$, we can utilize the inclusion
  $$
  \mathcal{R}(A_-)\subset (L^1)^n
  $$
  and the classical Hahn-Banach extension theorem to extend $\vec{U}_o$ to an element
  $$\vec{U}\in \big[(L^1)^n\big]^\ast=(L^\infty)^n$$
  such that
  $$
\laz   \vec U, \vec{V}\raz=\laz \vec U_o,  \vec{V}\raz\quad\forall\quad \vec{V}\in \mathcal{R}(A_-).
     $$
     Accordingly, if $\phi\in\mathcal{S}$, then
     $$
     \laz T,\phi\raz=\laz A^\ast_-\vec{U}_o,\phi\raz=     \laz   \vec U_o, \nabla^s_-\phi\raz=\laz \vec U,  \nabla^s_-\phi\raz=\laz [\nabla^s_-]^\ast\vec{U},\phi\raz,
 $$
 and hence
 $$
 T=[\nabla^s_-]^\ast\vec{U}\ \ \text{in}\ \ \mathcal{S}'.
 $$

\end{itemize}

\end{proof}

\subsubsection*{Fefferman-Stein decomposition \& Bourgain-Brezis question for John-Nirenberg space}\label{s522}
As a consequence of Theorem \ref{thm2}(iv), we surprisingly discover the coming-up-next assertion
is indeed
a resolution of the Bourgain-Brezis problem (cf. \cite[p.396]{BB}) asking for any function space $X$ between $W^{1,n}$ and $\BMO$ such that every $F\in X$ has a representation $$F=\sum_{j=1}^n R_j Y_j\ \ \text{where}\ \ (n-1,Y_j)\in\mathbb N\times L^\infty.
$$
As a subspace of the distribution space $\cs'/\mathcal P$, the above-searched space $X$ is nothing but $I_s[\mathring H^{s,1}_-]^\ast$ for every number $s\in (0,1)$. Moreover, the well known Fefferman-Stein decomposition (cf. \cite[Theorems 2\&3]{FS}) gives
$$
\BMO=L^\infty+\vec{R}\cdot\big(L^\infty\big)^{n}=L^\infty+I_s\big([\mathring H^{s,1}_-]^\ast\big)\quad\forall\quad s\in(0,1).
$$

\begin{theorem}\label{thm3}
	Let $s\in(0,1)$ and $n\in\nn$. Then
\begin{align}\label{bb-decom}
I_s[\mathring H^{s,1}_-]^\ast=\vec R\cdot(L^\infty)^n
\end{align}
in the sense of
$$f=I_s g \ \ \text{in}\ \ \cs'/\mathcal P\ \ \text{for some} \ \ g\in [\mathring H^{s,1}_-]^\ast$$
if and only if
$$
 \exists\ (Y_1,...,Y_n)\in \big(L^\infty\big)^n\ \ \text{such that}\ \
	f=\sum_{j=1}^nR_j Y_j \ \text{in}\ \cs'/\mathcal P\ \text{or}\ \BMO.
$$
Moreover, under $n\ge2$, it holds that
\begin{align}\label{embed}
W^{1,n}\subset I_s(L^{\frac ns})\subset I_s(L^{\frac ns,\infty})\subset I_s[\mathring H^{s,1}_-]^\ast \subset \BMO \ \text{in}\ \cs'/\mathcal P\ \text{or}\ \BMO.
\end{align}
As a consequence, when $n\ge2$, given any $Y_0\in I_s[\mathring H^{s,1}_-]^\ast$,
there exists $(Y_1,...,Y_n)\in \big(L^\infty\big)^n$ such that
$$\text{div}\big((-\Delta)^{-\frac12}Y_1,...,(\Delta)^{-\frac12}Y_n\big)=Y_0\ \text{in}\ \cs'/\mathcal P.$$
%
%
%
%
%
\end{theorem}

\begin{proof}
Let us first validate $$I_s[\mathring H^{s,1}_-]^\ast\subset \vec R\cdot(L^\infty)^n$$ in \eqref{bb-decom}.
If $f=I_s g$ for some $g\in[\mathring H^{s,1}_-]^\ast$, then $$g\in\cs'\subset \cs'/\mathcal P$$ and hence
$I_sg$ is a well defined distribution in $\cs'/\mathcal P$. Applying Theorem \ref{thm2}(iv), we write
$$g=[\nabla^s_-]^\ast \vec Y\ \ \text{in}\ \ \cs'$$
for some $$\vec Y=(Y_1,\dots, Y_n)\in (L^\infty)^n.
$$ Consequently, for any $\phi\in\cs_\infty$, upon letting $\psi=I_s\phi$, we have
\begin{align*}
\laz f,\phi\raz
=\laz I_s g,\phi\raz
=\laz g, I_s\phi\raz
=\laz [\nabla^s_-]^\ast \vec Y,\, \psi\raz
=-\sum_{j=1}^n\laz  Y_j, \nabla^s_j\psi\raz
=-\sum_{j=1}^n\laz  Y_j, R_j\phi\raz
=\sum_{j=1}^n \laz R_j Y_j,\phi\raz,
\end{align*}
which gives that
$$
f=\sum_{j=1}^n R_j Y_j \quad \text{in}\quad \cs'/\mathcal P.
$$
Due to the density of $\cs_\infty$ in $H^1$ (cf. \cite{Adams}), the above identities also implies
$$
f=\sum_{j=1}^n R_j Y_j \quad \text{in}\quad [H^1]^\ast=\BMO.
$$

Now, we are about to show the part $$\vec R\cdot(L^\infty)^n \subset I_s[\mathring H^{s,1}_-]^\ast$$ in \eqref{bb-decom}.
Given any $$\vec Y=(Y_1,\dots, Y_n)\in (L^\infty)^n,$$
we utilize Theorem \ref{thm2}(iv) to derive that
$$
[\nabla^s_-]^\ast \vec Y=-\sum_{j=1}^n (-\Delta)^\frac s2 R_jY_j= g \ \ \text{in}\ \ \cs'\ \ \text{for some} \ \ g\in [\mathring H^{s,1}_-]^\ast,
$$
which in turn gives that
$$\vec R\cdot \vec Y=\sum_{j=1}^n R_jY_j=-I_sg \ \ \text{in}\ \ \cs'/\mathcal P,$$
that is,
$$
\vec R\cdot \vec Y\in I_s[\mathring H^{s,1}_-]^\ast,
$$
as desired. So, we complete the proof of \eqref{bb-decom}.

Next, we show \eqref{embed}. Due to the density of $\cs_\infty$ in $H^1$, it suffices to show its validity in $\cs'/\mathcal P$.
To do so, let us begin with verifying $$
W^{1,n}\subset I_s(L^\frac ns)\qquad\forall\ n\ge2.$$
For any $f\in W^{1,n}$, since $W^{1,n}\subset\BMO$ (cf. \cite{BB}) and $\BMO\subset\cs_s'$ (cf. Lemma \ref{lem-add2}), we know from Definition \ref{defn1}(i) that $(-\Delta)^\frac s2 f$ is a well-defined distribution in $\cs'$ and
$$\laz(-\Delta)^\frac s2 f,\psi\raz=\laz f, (-\Delta)^\frac s2 \psi\raz\qquad\forall\ \psi\in\cs.$$ 
 In particular, if $\psi\in\cs_\infty$, then the Fourier transform gives that
\begin{align*}
(-\Delta)^\frac s2 \psi= -\nabla\cdot\lf(\nabla I_{2-s}\psi\r) 
\end{align*}
holds pointwisely, so that we can use the integral by parts formula and 
the continuity of the mapping $I_{1-s}:\, L^{\frac n{n-s}}\to L^{\frac n{n-1}}$ (cf. \cite{Aduke}) to derive
\begin{align*}
\lf|\laz(-\Delta)^\frac s2 f,\psi\raz\r|
&=\lf|-\int_{\rn} f \lf(\nabla\cdot\lf(\nabla I_{2-s}\psi\r)\r)\,dx\r|\\
&=\lf |\int_\rn (\nabla f)\cdot (\nabla I_{2-s}\psi)\,dx\r|\\
&\ls\int_\rn |\nabla f| |I_{1-s}\psi|\,dx\\
&\ls \|\nabla f\|_{L^n}\|I_{1-s}\psi\|_{L^{\frac n{n-1}}}\\
&\ls \|f\|_{W^{1,n}}\|\psi\|_{L^{\frac n{n-s}}}.
\end{align*}
Further, using the density of $\cs_\infty$ in $L^{\frac n{n-s}}$ (cf. the proof of Lemma \ref{lem-add3}(iii)), we know that 
$$(-\Delta)^\frac s2 f\in [L^{\frac n{n-s}}]^\ast=L^{\frac ns}.$$
Thus, understood in the sense of distributions, it holds 
$f=I_s((-\Delta)^\frac s2 f)$ in $\cs'/\mathcal P$ and $W^{1,n}\subset I_s(L^\frac ns)$. 


Also, since
$$I_s(L^\frac ns)\subset I_s(L^{\frac ns, \infty})$$ is obvious, we secondly validate $$L^{\frac ns, \infty}\subset [\mathring H^{s,1}_-]^\ast.
$$
Upon observing that
$L^{\frac ns, \infty}$  exists as the dual of the Lorentz space
$$L_{\frac n {n-s}, 1}=\lf\{f\ \text{is measurable on }\ \rn:\, \|f\|_{L_{\frac n {n-s}, 1}}=\int_0^\infty |\{x\in\rn:\, |f(x)|>t\}|^{\frac{n-s}{n}} \,dt<\infty\r\},$$
we use the H\"older inequality for weak Lebesgue space (cf. \cite[Exercise~1.1.15]{Grafakos}) and \cite[(1.3)]{Sp18} to derive
that for any $(f, g)\in L^{\frac ns, \infty}\times\cs$
one has
\begin{align*}
|\laz f,g \raz|
&\le \|f\|_{L^{\frac ns, \infty}}\|g\|_{L_{\frac n {n-s}, 1}} \ls \|f\|_{L^{\frac ns, \infty}}\|\nabla^s_-g\|_{L^1}
\ls \|f\|_{L^{\frac ns, \infty}} [g]_{H^{s,1}_-}.
\end{align*}
Combining this and density of $\cs$ in $\mathring H^{s,1}_-$ shows that any $f\in L^{\frac ns, \infty}$ can induce a bounded linear functional on $H^{s,1}_-$. This proves $$L^{\frac ns, \infty}\subset[\mathring H^{s,1}_-]^\ast.$$

Finally, due to $$Y_0\in I_s[\mathring H^{s,1}_-]^\ast=\vec R\cdot(L^\infty)^n,$$ we can find a vector-valued function
$$
\vec{g}=(g_1,...,g_n)\in \big(L^\infty\big)^n
$$
such that
$$
Y_0=\sum_{j=1}^nR_jg_j=\nabla\cdot\big((-\Delta)^{-\frac12}\vec{g}\big)=\text{div}\Big((-\Delta)^{-\frac12}g_1,...,(-\Delta)^{-\frac12}g_n\Big)\quad \text{in}\quad \cs'/\mathcal P.
$$
\end{proof}

\noindent{\bf Acknowledgement}: We would like to thank D. Spector for his constructive comments on the original version of this paper.

\end{document}